\documentclass[10pt,a4paper]{amsart}
\usepackage{amsmath}
\usepackage{amsthm}
\usepackage{amssymb}
\usepackage{mathrsfs} 
\usepackage{exscale}
\usepackage{verbatim}
\usepackage{graphicx}
\usepackage{textcomp}
\usepackage{array}
\usepackage[a4paper,breaklinks]{hyperref}

\allowdisplaybreaks
\theoremstyle{definition}
    \newtheorem{definition}{Definition}
    \newtheorem*{definition*}{Definition}
    \newtheorem{example}[definition]{Example}
    \numberwithin{definition}{subsection}
\theoremstyle{plain}
    \newtheorem{lemma}[definition]{Lemma}
    \newtheorem{proposition}[definition]{Proposition}
    \newtheorem{theorem}[definition]{Theorem}
    \newtheorem*{theorem*}{Theorem}
    \newtheorem{corollary}[definition]{Corollary}

    \newtheorem*{claim*}{Claim}

\theoremstyle{remark}

    \newtheorem{remark}[definition]{Remark}

\DeclareMathOperator{\Hom}{Hom}

\DeclareMathOperator{\tr}{tr}
\DeclareMathOperator{\cop}{c}
\DeclareMathOperator{\rk}{rk}

\DeclareMathOperator{\im}{im}
\DeclareMathOperator{\Rot}{Rot}

\DeclareMathOperator{\stab}{stab}
\DeclareMathOperator{\Rep}{Rep}
\DeclareMathOperator{\sgn}{sgn}
\DeclareMathOperator{\sign}{sign}

\DeclareMathOperator{\PSL}{PSL}
\DeclareMathOperator{\SL}{SL}
\DeclareMathOperator{\GL}{GL}
\DeclareMathOperator{\Sp}{Sp}

\DeclareMathOperator{\Sym}{Sym}
\DeclareMathOperator{\On}{O}
\DeclareMathOperator{\Str}{Str}
\DeclareMathOperator{\pt}{pt}
\DeclareMathOperator{\h}{H}

\renewcommand{\c}{\cop}

\renewcommand{\phi}{\varphi}
\renewcommand{\rho}{\varrho}

\newcommand{\C}{\mathbb{C}}
\newcommand{\Z}{\mathbb{Z}}
\newcommand{\N}{\mathbb{N}}
\newcommand{\R}{\mathbb{R}}
\newcommand{\D}{\mathbb{D}}

\newcommand{\Mca}{\mathcal{M}}
\newcommand{\Xca}{\mathcal{X}}
\newcommand{\Dca}{\mathcal{D}}

\newcommand{\id}{\rm{id}}

\newcommand{\Rott}{\widetilde \Rot}

\newcommand{\X}{\Xca}

\newcommand{\M}{\Mca}

\begin{document}
\bibliographystyle{plain}

\title[Fenchel-Nielsen coordinates]{Fenchel-Nielsen Coordinates for Maximal Representations}
\author{Tobias Strubel}
\date{\today}
\begin{abstract} We develop Fenchel-Nielsen coordinates for representations of surface groups into $\Sp(2n,\R)$ with maximal Toledo invariant. Analogous to classical Fenchel-Nielsen coordinates on the Teichm\"uller space they consist of a parametrization of representations of the fundamental group of a pair of pants and a careful investigation of the gluing. 

As applications we obtain results for non-closed surfaces, which have been known only for closed surfaces before: we count the number of connected components of their representation space and prove continuity for the limit curve for a certain type of representations.
\end{abstract}

\maketitle

\setcounter{tocdepth}{1}
\tableofcontents
\newcolumntype{C}[1]{>{\centering\arraybackslash}p{#1}}

\section{Introduction}\label{SecIntro}
 
 \subsection{Representations of Surface Groups and Coordinates}
 Representations of surface groups into Lie groups have received considerable interest during the last years. One starting point for their study has been the fact that the  Teichm\"uller space $\mathcal T(\Sigma_g)$ of a closed oriented surface $\Sigma_g$ can be realized as a connected component of the representation space $\Rep(\pi_1(\Sigma_g),\PSL(2,\R))$ via the holonomy representation. Some of the properties of holonomy representations have been proven for representations into other Lie groups: Hitchin representations and maximal representations (defined below) are  discrete and faithful and the mapping class group acts properly on the respective representations. Although there exist many approaches for coordinates for the classical Teichm\"uller space, so far no coordinates where known for maximal representations. (For Hitchin representations there exist coordinates generalizing Penner coordinates.)

In this paper we introduce and study Fenchel-Nielsen coordinates
on the space of maximal representations into $\Sp(2n,\R)$. Our two main results are an explicit description of the
character variety of maximal representations
of a pair of pants (Theorem \ref{MainThmIntro}), and the determination for general surfaces with
boundary of the number of connected components
of the space of maximal representations (Theorem \ref{ThmConnComp}).

Being concerned with representation of surfaces with boundary we need a definition for this case of the Toledo invariant having the right additivity properties. 

We begin with the classical definition. Let $G$ be a Hermitian Lie group of non-compact type and $\Gamma_g$ be the fundamental group of a closed oriented surface $\Sigma_g$ of genus $g\geq 2$. Using the identification $\h^2(\pi_1(\Sigma),\R)\simeq \h^2(\Sigma_g,\R)$ the \emph{Toledo invariant} of a representation $\varrho:\Gamma_g\rightarrow G$ is defined as follows:
 \[
   T_\varrho:=\langle \varrho^*(\kappa_G),[\Sigma_g]\rangle,
 \]
 where $\kappa_G\in \h_c^2(G,\R)$ is the K\"ahler class (defined Section \ref{SecToledo} below) of $G$ and $[\Sigma_g]\in \h_2(\Sigma_g,\R)$ is the orientation class.
 
Let now $\Sigma$ be a surface of finite type with boundary. Then $\h_c^2(\Sigma,\R)=0$, hence the definition above fails. But one can use bounded cohomology to circumvent this problem.
We have
\[
   \h^2_b(\pi_1(\Sigma),\R)\simeq \h_b^2(\Sigma,\R)\simeq \h^2_b(\Sigma,\partial \Sigma,\R),   
 \]
 which is an infinite dimensional vector space.
 
Then the pullback $\varrho^*\kappa_G^b$ of the bounded K\"ahler class $\kappa_G^*$ can be seen as an element of $\h^2_b(\Sigma,\partial \Sigma,\R)$ and again we define
 \[
   T_\varrho:=\langle\varrho^*\kappa_G^b,[\Sigma,\partial \Sigma]\rangle.
 \]
 We refer the reader to \cite{Surface} for all details. The Toledo invariant can not take arbitrary values. In fact
\[
|T_\varrho|\leq |\chi(\Sigma)|\rk \Dca,
\]
where $\Dca$ is the bounded symmetric associated with $G$. This inequality is sharp and we can define:
\begin{definition}
  A representation $\varrho:\pi_1(\Sigma)\rightarrow G$ into a Hermitian Lie group is \emph{maximal}, if its Toledo invariant is maximal, i.e.
  \[
    T_\varrho=|\chi(\Sigma)|\rk\Dca.
  \]
\end{definition}
The Toledo invariant is conjugation invariant, hence it descends to  
 \[
    \Rep_{max}(\pi_1(\Sigma),G):=\Hom_{max}(\pi_1(\Sigma),G)/G,
  \]
  where $\Hom_{max}(\pi_1(\Sigma),G)$ is the space of maximal representations on which $G$ acts by conjugation on it.

Furthermore the Toledo invariant is its additivity. Let $\Sigma=\Sigma_1\cup \Sigma_2$ be a surface cut into $\Sigma_1$ resp. $\Sigma_2$ along a simple closed curve. A representation $\varrho:\pi_1(\Sigma)\rightarrow G$ is said to be \emph{glued} from $\varrho_1:\pi_1(\Sigma_1)\rightarrow G$ and $\varrho_2:\pi_1(\Sigma_2)\rightarrow G$ if $\varrho|_{\pi_1(\Sigma_i)}=\varrho_i$. If $\varrho$ is glued from $\varrho_1$ and $\varrho_2$, then
\[
  T_\varrho=T_{\varrho_1}+T_{\varrho_2}.
\]
In particular a representation is maximal if and only if it is glued from maximal representations. For more details see Section \ref{SecGlueSp2nR}  and Theorem \ref{PropPropToledo}.

A strong motivation for the study of maximal representations is Goldman's theorem, which says that maximality of representations into $\PSL(2,\R)$ characterizes holonomy representations of marked hyperbolic structures on the underlying surface (\cite{GoldmanDiss}). 
  
Maximal representations have been studied using quite different techniques, such as bounded cohomology \cite{Surface,Anosov,AnnaMCG,LimitCurves,HS09,GW09,GWNew} and Higgs bundle techniques \cite{RepSymplHiggs,GothenComp,DefMaxRep}. For an introduction and an overview we refer the reader to \cite{BIWSl2}.

Before we turn to our main results, we briefly recall classical Fenchel-Nielsen coordinates on  Teichm\"uller space. They are enabled to the fact that each orientable surface of negative Euler characteristic can be decomposed into pairs of pants (see Figure \ref{FigPoP} and Figure \ref{DecompPoP}) and that marked hyperbolic structures are easy to describe on a pair of pants. 
\begin{figure}[ht]
\centering
\includegraphics{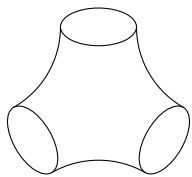} 
\caption{$\Sigma_{0,3}$ or Pair of Pants}
\label{FigPoP}
\end{figure}

\begin{figure}[ht]
\centering
\includegraphics[scale=0.7]{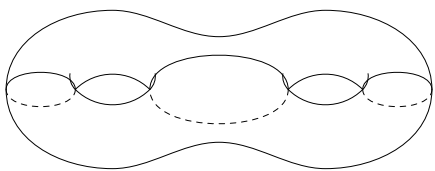} 
\caption{Decomposition of $\Sigma_2$ into two pairs of pants.}
\label{DecompPoP}
\end{figure}

Namely they are determined uniquely by the length of their boundaries (\emph{length parameter}). Gluing of two surfaces along geodesic boundaries is unique up to a rotation along the boundary, which yield another parameter (\emph{twist parameter}). It turns out that length and twist parameter determine the hyperbolic structure uniquely. For details on Fenchel-Nielsen coordinates see e.g. \cite{Abikoff} or \cite{Jost}.

To obtain similar coordinates for (maximal) representations we have to understand the representations of the fundamental group of the pair of pants on one hand and the gluing construction for representations on the other hand.
 
In the sequel we restrict ourselves to representations into the symplectic group $\Sp(2n,\R)$.

 To state the main theorem for generalized length parameters, denote by $B$ the set of matrices in $\GL(n,\R)$ whose eigenvalues have absolute value strictly less than $1$. Its closure $\bar B$ is the set of matrices whose eigenvalues have absolute value less or equal to $1$. We define
  \begin{align*}
     R:= \{(X_1,X_2,X_3)\in \bar B^3 | & X_3(X_2^\top)^{-1}X_1 \text{ is symmetric} \\ &\text{and positive definite}\}.
  \end{align*}
  Note that $\On(n)$ acts diagonally by conjugation on $R$. We denote the fundamental group of a pair of pants by
\[
  \Gamma_{0,3}=\langle C_3,C_2,C_1| C_3C_2C_1\rangle
\]
(see the general notation in Section \ref{SecTeichmRep}).

Here and in the sequel we write elements of $\Sp(2n,\R)$ as
\[
     g=\left(
\begin{array}{cc}
	A&B\\
	C& D
\end{array}\right),
\]
where $A$, $B$, $C$ and $D$ are real $n\times n$-matrices satisfying the following relations:
\begin{equation}
  A^\top D-C^\top B=I,\quad A^\top C=C^\top A,\quad D^\top B=B^\top D. \label{inspn}
\end{equation}

The following theorem was inspired by  \cite[Ch. 10]{LMcS09}, where results of \cite{FG06} are presented. Compare also with \cite{GoldmanConvex}. We will prove it in Section \ref{Sec03}.

\begin{theorem}\label{MainThmIntro}
There exists a homeomorphism 
\[
  f:R/\On(n)\rightarrow \Rep_{max}(\Gamma_{0,3},\Sp(2n,\R)).
\]
  It is induced by the map $\bar f :R\rightarrow \Rep(\Gamma_{0,3},\Sp(2n,\R))$ which assigns to $(X_1,X_2,X_3)\in R$ the representation $\varrho=\bar f(X_1,X_2,X_3)$ of $\Gamma_{0,3}$ into $\Sp(2n,\R)$ defined by
\begin{align*}
  \varrho(C_1):=c_1&=\left(\begin{array}{cc}
	X_1&0\\
	X_1+X_2^{-1}X_3^\top &(X_1^\top)^{-1}
\end{array}\right)\\
\varrho(C_2):=c_2&=\left(
\begin{array}{cc}
	-X_3^{-1}X_1^\top-X_2-(X_2^\top)^{-1}&X_2 + X_3^{-1}X_1^\top\\
	-X_3^{-1}X_1^\top - (X_2^\top)^{-1}&X_3^{-1}X_1^\top
\end{array}\right)\\
\varrho(C_3):=c_3&=\left(
\begin{array}{cc}
	(X_3^\top)^{-1}&-(X_3^\top)^{-1}-X_1^{-1}X_2^\top\\
	0& X_3
\end{array}\right).
\end{align*}

\end{theorem}
This result can be generalized for maximal representations into any Hermitian Lie group of tube type, see \cite[Prop. 3.4.4]{S11}.

The $X_i$ can be seen as \emph{generalized length parameters}. We discuss properties of these parameters in Section \ref{Sec2} below.

Above we introduced the gluing for representations. The following theorem distinguishes the cases in which we can glue representations and introduces the corresponding twist parameter.

\begin{theorem}\label{PropGlueSp2nRIntro}
 Let $\varrho:\pi_1(\Sigma) \rightarrow \Sp(2n,\R)$ and $\bar \varrho:\pi_1(\bar \Sigma) \rightarrow \Sp(2n,\R)$ be (non-necessarily distinct) maximal representation with distinct boundary components $C\subset \Sigma$ and $\bar C\subset \bar \Sigma$. 
\begin{enumerate}
\item
We can conjugate $\varrho$ and $\bar \varrho$ such that
  \begin{equation}\label{EqcNormalform1Intro}
  c:=\varrho(C)=\left(\begin{array}{cc}
	X&0\\
	X+(X^\top)^{-1}S &(X^\top)^{-1}
\end{array}\right)
\end{equation} 
and 
\begin{equation}
\bar c:=\bar \varrho(\bar C)=\left(
\begin{array}{cc}
	(\bar X^\top)^{-1}&-(\bar X^\top)^{-1}-\bar S\bar X \\
	0& \bar X
\end{array}\right)\end{equation}
with $X$ and $\bar X$ invertible and $S$ and $\bar S$ symmetric positive definite. 
\item 
The representation classes $[\varrho]$ and $[\bar \varrho]$ can be glued along $C$ and $\bar C$ if and only if $\varrho(C)^{-1}$ and $\bar \varrho(\bar C)$ are conjugate in $\Sp(2n,\R)$.

 \item Suppose $X$ and $\bar X$ contracting, i.e. their eigenvalues have absolute value strictly less than $1$. 
Then $\bar c$ and $c^{-1}$ are conjugate in $\Sp(2n,\R)$ if and only there exists $G\in \GL(n,\R)$ such that $X^\top=G\bar X G^{-1}$.
\item It $X$ or $\bar X$ has an eigenvalue of absolute value $1$, then $\bar c$ and $c^{-1}$ are not conjugate in $\Sp(2n,\R)$.
\end{enumerate}
\end{theorem}
If $c^{-1}=g\bar cg^{-1}$ then $G$ is the \emph{generalized twist parameter}. Theorem \ref{PropGlueSp2nRIntro} follows from Theorem \ref{PropGlueSp2nR}, where we also make $g$ explicit.

Using Theorems \ref{MainThmIntro} and \ref{PropGlueSp2nRIntro} one can obtain coordinates for spaces of all maximal representation. We present some examples and a general description of the coordinates in Section \ref{SecMoreParam}.

From these coordinates we can deduce some applications. We can count connected components of spaces of maximal representations of fundamental groups of surfaces with at least one boundary component:
\begin{theorem}\label{ThmConnComp}
  If $m\geq 1$ then $\Rep_{max}(\Gamma_{g,m}, \Sp(2n,\R))$ has $2^{2g+m-1}$ connected components.
\end{theorem}

So far connected components have been counted only for representations of fundamental groups of closed surfaces \cite{RepSymplHiggs,GothenComp,DefMaxRep}. In contrast to the results there, we have no special connected components for $\Sp(4,\R)$.

We will prove Theorem \ref{ThmConnComp} in Section \ref{SecApp}. There we also present a class of representations which are Anosov and for which the limit curve is continuous. 

An important ingredient for the proof of Theorem \ref{MainThmIntro}, which might also be of independent interest, is:

\begin{theorem}\label{ThmFormulaIntro}
   Let $G$ be a Hermitian Lie group of tube type. 
   Let $\varrho:\Gamma_{0,3}\rightarrow G$ be a representation and denote $c_i:=\varrho(C_i)$. Assume that each $c_i$ has a fixed point $y_i$ in the Shilov boundary $\check S$. Then we can express the Toledo invariant as follows:
   \begin{equation}\label{Formula}
      T_{\varrho}=\frac{1}{2}\big(\beta(y_1,y_2,y_3)+\beta(y_1,c_1\cdot y_3,y_2)\big),
   \end{equation}
   where $\beta$ denotes the Maslov index.
\end{theorem}
We introduce the Shilov boundary and the Maslov index in Section \ref{SecBSD} and prove this formula in Section \ref{SecFormula} below.

\subsection{More on Coordinates}\label{Sec2}

The generalized length parameters $X_i$ from Theorem \ref{MainThmIntro} are related to the length parameters for the Teichm\"uller space introduced above as follows: consider a hyperbolic structure on a pair of pants and denote by $h$ a \emph{hyperbolization} (i.e. a monodromy representation $h:\pi_1(\Sigma)\rightarrow \PSL(2,\R)$) for it. Then the length of each boundary component $C_i$ is equal to the translation length of $h(C_i)$, where for a metric space $X$ and an isometry $g$ of $X$ the \emph{translation length} is defined as
\[
  \tau_X(g):=\inf_{x\in X}d(x,gx).
\]
For a hyperbolic element $g$ of $\PSL(2,\R)$ with eigenvalues $\lambda$ and $\lambda^{-1}$, the translation length is
\[
  \tau_\D(g)=2\big|\log |\lambda|\big|.
\]
For $G=\PSL(2,\R)$, the matrices $X_i$ are real numbers and eigenvalues of the $\c_i$. Therefore they determine the boundary length of $C_i$.

The generalized length parameters $X_i$ had to satisfy  $X_3(X_2^\top)^{-1}X_1$ symmetric and positive definite. A direct calculation shows that $c_3c_2c_1=I$ if and only if $X_3(X_2^\top)^{-1}X_1$ is symmetric. Its signature \footnote{The signature $\sgn$ of a symmetric matrix is the number of positive eigenvalues minus the number of negative eigenvalues.} determines the Toledo invariant. Indeed with Formula \eqref{Formula} we get
  \[
    T_\varrho=\frac{1}{2}(n+\sgn X_3(X_2^\top)^{-1}X_1).
  \]
We will use this to write down also a certain type of non-maximal representations explicitly, see Proposition \ref{Prop03} below.

 With the formulas from Theorem \ref{MainThmIntro} any triple $(X_1,X_2,X_3)\in \GL(n,\R)^3$ with $X_3(X_2^\top)^{-1}X_1$ symmetric positive definite yields a maximal representation.
The $X_i$ are required to be in $\bar B\subset \GL(n,\R)$ to give unique parameters (c.f. Remark \ref{RemNonUnique}).

Theorem \ref{PropGlueSp2nRIntro} also has a geometric interpretation for $G=\Sp(2,\R)=\SL(2,\R)$. It  corresponds to the fact, that one can glue hyperbolic surfaces along two geodesic boundaries if and only if they have the same length. Indeed, let $\Sigma_1$ and $ \Sigma_2$ be surfaces with boundary components $C_1$ resp. $C_2$ and $\varrho_1$ and $ \varrho_2$ hyperbolizations. Define $c:=\varrho_1(C_1)$ and $\bar c:= \varrho_2(C_2)$  and assume that they are in the form of  Theorem \ref{PropGlueSp2nRIntro}. Then the lengths of the corresponding boundaries are equal to the translation length of $c$ resp. $\bar c$. The translation length is uniquely determined by the eigenvalues of $c$ and $\bar c$.   
Then one can glue along these boundary components if and only if $X$ and $\bar X$ are equal and their absolute values is different from $1$, i.e. $c$ and $\bar c$ are hyperbolic and have the same translation length. But this is precisely the statement of Theorem \ref{PropGlueSp2nRIntro}.

\subsection{Guide for the reader}
This article is structured as follows:

Section \ref{SecPrelims} contains all necessary notions and definitions. Since it contains no new results, the reader may skip it and come back to it if necessary.

In Section \ref{SecFormula} we introduce the K\"ahler class which we used to define the Toledo invariant as well as the Souriau index and prove Formula \ref{Formula}. 

Theorem \ref{MainThmIntro} will be proven in Section \ref{Sec03}. 

In Section \ref{SecGlue} we prove Theorem \ref{PropGlueSp2nRIntro}. 

The most general statement on the parameters as well as a more detailed discussion on some examples of representations can be found in Section \ref{SecMoreParam}.

Finally we prove the Applications in Section \ref{SecApp}.

\textbf{Acknowledgement}\\
This text is a part of the authors PhD thesis. I thank Marc Burger for the supervision of the thesis as well as Olivier Guichard for detailed feedback concerning earlier version of the thesis. 

I am grateful to Tobias Hartnick, Alessandra Iozzi, Anke Pohl and Anna Wienhard for many useful conversations and remarks concerning earlier versions of this text.

I thank the Institut Henri Poincar\'e for their hospitality.

I was supported by SNF grants PP002-102765 and 200021-127016.
\section{Preliminaries}\label{SecPrelims}

In this section we introduce the terminology used throughout this paper. 
\subsection{Notations}\label{SecTeichmRep}

We denote by $\Sigma_{g,m}$ an oriented surface of genus $g$ with $m$ boundary components and  by $\Gamma_{g,m}$  its fundamental group. A surface of genus $g$ without boundary is denoted by $\Sigma_g$, its fundamental group by $\Gamma_g$. 

We fix the following standard presentation for $\Gamma_{g,m}$
\begin{align*}
  \Gamma_{g,m}=\langle & A_1,B_1,\ldots,A_g,B_g,C_1,\ldots, C_{m} | \\
   & [A_g,B_g]\ldots [A_1,B_1]C_m\ldots C_1=e \rangle.
\end{align*} 
The $C_j$ correspond to loops around boundary components. We call the $A_i$, $B_i$ and $C_j$ the \emph{standard generators}.

Let $G$ be a topological group. We denote by $ \Hom(\Gamma_{g,m},G)$ the set of homomorphism from $\Gamma_{g,m}$ into $G$. The group $G$ acts by pointwise conjugation on $ \Hom(\Gamma_{g,m},G)$. The quotient space with respect to this action is
\[
  \Rep(\Gamma_{g,m},G)=\Hom(\Gamma_{g,m},G)/G,
\]
the \emph{representation variety}. We equip $\Hom$ and $\Rep$ with a topology. If $m\neq 0$, then $\Gamma_{g,m}$ is a free group of degree $2g+m-1$ and $\Hom(\Gamma_{g,m},G)$ can be identified with  $G^{2g+m-1}$ and we can carry over its topology. If $m=0$, then we have only one relation in $\Gamma_g$, hence $\Hom(\Gamma_g,G)$ is a quotient with respect to this relation and we can use the quotient topology.  If $G$ is algebraic, then so is $ \Hom(\Gamma_{g,m},G)$. We also equip $\Rep(\Gamma_{g,m},G)$ with the quotient topology.

\subsection{Hermitian Symmetric Spaces of Tube Type and Jordan Algebras}\label{SecBSD}
The main references for this section are \cite{FK} and \cite{Helgason}.
\begin{definition}
  A Lie group $G$ is \emph{Hermitian} if the associated symmetric space $\mathcal X$ has an invariant complex structure. The symmetric space $\mathcal X$ is called \emph{Hermitian symmetric space}. The group $G$ and its symmetric space are of \emph{tube type}, if $\mathcal X$ is biholomorphic to a tube domain
  \[
     T_\Omega=V \oplus i\Omega\subset V^\C,
  \]
where $V$ is a finite dimensional real vector space and $\Omega \subset V$ a convex open cone.
\end{definition}

Every symmetric space can be decomposed into an euclidian factor, simple compact factors and simple non-compact factors \cite[Prop. 4.2]{Helgason}.

\begin{definition}
  A symmetric space is of \emph{non-compact type} if it has no compact factors.
\end{definition}

\begin{example}
  The group $\Sp(2n,\R)$ is a Hermitian Lie group of tube type. Indeed, let $V=\Sym_n(\R)$  be the vector space of real symmetric matrices and $\Omega\subset V$ be the cone of symmetric positive definite matrices. Define
  \[
  \Dca:=\{ X\in V^\C|~ \|X\|<1\},
\]
where $\| \cdot \|$ is the spectral norm on $V^\C=\Sym_n(\C)$. Equipped with its Bergman metric $\Dca$ is a bounded symmetric domain \cite[Ch.X]{FK}. 

Note that $\Dca$ is \emph{centered}, i.e. it is invariant under multiplication with elements of $S^1 \subset \C$. It is possible to realize each bounded symmetric domain as a centered domain.

The domain $\Dca$ is biholomorphically equivalent to
\[
  T_\Omega:=V\oplus i\Omega =\{X\in \Sym_n(\C)| \im X \text{ positive definite }\},
\]
where explicit maps are given by the Cayley transform and its inverse:
\[
  c:\begin{cases}
     \Dca\rightarrow T_\Omega \\
     Z\mapsto i(I+Z)(I-Z)^{-1}
  \end{cases}
\]
and
\[
  p: \begin{cases} T_\Omega\rightarrow \Dca\\
 z\mapsto (Z-iI)(Z+iI)^{-1}.
  \end{cases}.
\]

The group $\Sp(2n,\R)$ acts on $T_\Omega$
via
\[
  \left(
    \begin{array}{cc}
   A & B\\ C&D 
  \end{array}
\right)X=(AX+B)(CX+D)^{-1}.
\]
\end{example}
Therefore $\Sp(2n,\R)$ is of tube type.

Recall that 
\begin{equation}\label{EqInverse}
  \left(
    \begin{array}{cc}
   A & B\\ C&D 
  \end{array}
\right)^{-1}=\left(
    \begin{array}{cc}
   D^\top & -B^\top\\ -C^\top &A^\top 
  \end{array}
\right)
\end{equation}

The bounded symmetric domain $\Dca$ associated with a Hermitian Lie group $G$ is a bounded subset of a complex vector space, hence we can define the topological compactification $\bar \Dca$.  Its boundary decomposes into $G$-orbits and there is a unique closed one. This is the \emph{Shilov boundary} $\check S$. For more details on the topological compactification and the Shiov boundary we refer to \cite{Wolf} and \cite{Clerc}.

\begin{example}\label{ExShilov}
  For $G=\Sp(2n,\R)$ the Shilov boundary is:
\[
  \check S=\{X\in \Sym_n(\C)| ~|\lambda |=1 \text{ for all eigenvalues }\lambda \text{ of }X\}.
\]
\end{example}

\begin{remark}\label{RemVopendense}
From this expression and the definition of the Cayley transform $c$ it is immediately clear that $c$ is not defined on whole $\check S$, but only on the set of $X\in \check S$ such that $X-I$ is invertible. In the view of Definition \ref{DefTransverse} below this the case if and only if $X\pitchfork I$. The Cayley transform maps this set surjectivly to $V$. In particular 
\[
  p(V)=\{ X\in \check S| X\pitchfork I  \} 
\]
is an open and dense subset of $\check S$.
\end{remark}
\begin{example}\label{ExShilovLagrangians}
The Shilov boundary for $\Sp(2n,\R)$ can be identified with the the space of Lagrangians  $L(\R^{2n})$  of a symplectic vector space $\R^{2n}$. Indeed the Lagrangian subspace spanned by the standard basis vectors $e_{n+1},\ldots, e_{2n}$ is stabilized by the subgroup 
 \[
   \left\{ \left(\begin{array}{cc} A &0\\ C &D\end{array} \right)\in \Sp(2n,\R)  \right\}
 \]
 and this is also the stabilizer of the point $0\in V$. In both cases $\Sp(2n,\R)$ acts transitively, hence one can identify $\check S$ and $L(\R^{2n})$ as homogeneous $\Sp(2n,\R)$-spaces.\\

\end{example}

The relation between all the objects introduced so far becomes more conceptual if one introduces \emph{Jordan algebras}. A \emph{Jordan algebra} is a commutative algebra $V$ which satisfies 
\[
  x^2(xy)=x(x^2 y),\qquad \forall x,y \in V.
\]
We refer the reader to \cite{FK} for an introduction to Jordan algebras as well as to \cite{Bertram} for the relation between bounded symmetric domains and Jordan algebras.

A \emph{Jordan frame} for the Jordan algebra $\Sym_n(\R)$ is a set of idempotents $\{c_i\}$ such that
\begin{itemize}
  \item $c_ic_j=0$ for all $i\neq j$,
  \item $\sum c_i=I$,
  \item no $c_i$ can be written as the sum of two idempotents.
\end{itemize}
The easiest Jordan frame for $V=\Sym_n(\R)$ is the set of matrices with one $1$ on the diagonal and zero elsewhere. All other Jordan frames are conjugate under $\On(n)$ to this one. Every element $ Z\in \Sym_n(\R)$ can be written as
\[
  Z=\lambda_1 c_1+\cdots +\lambda_nc_n,
\]
where the $\lambda_i$ are real numbers and $\{c_i\}$ a Jordan frame. The \emph{Jordan algebra determinant} is defined as
\[
  {\det}_V(Z)=\prod\lambda_i.
\]

\begin{definition}\label{DefTransverse}
  Two elements $X,Y$ of $V^\C=\Sym_n(\C)$ are \emph{transverse} (denoted by $X\pitchfork Y$) if $\det(X-Y)\neq 0$.
\end{definition}
For various characterizations of transversality in the Shilov boundary see \cite[Proposition 14]{HS09}. The Cayley transform preserves transversality (\cite[p.200 ff]{FK}).

\begin{definition}\label{DefMaslov}
   Fix a Jordan frame $\{c_i\}$ and define for $k\in \{0,\ldots,n\}$ 
  \[
    I_k:=\sum_{j=1}^k c_j-\sum_{j=k+1}^n c_j.
  \]Let $(Y_1,Y_2,Y_3)$ be a triple of pairwise transverse points in the Shilov boundary $\check S$. 
  Following \cite{CO2} we define the \emph{Maslov index} 
\[ \beta(Y_1,Y_2,Y_3)= 2k-n
\]
 if there exists $g\in G$ such that $g(Y_1,Y_2,Y_3)=(-e,-iI_k,e)$.
\end{definition}
For the Jordan frame for $V=\Sym_n(\R)$ defined above we have
\[
  I_k= \left(
    \begin{array}{cc}
   1_k & \\ & -1_{n-k}
  \end{array}
\right),
\]
where $1_k$ is the $k\times k$ unit matrix.
The Maslov index is skew-symmetric, $G$-invariant and it classifies $G$-orbits of triples of pairwise transverse points in the Shilov boundary. It can be defined for arbitrary triples in the Shilov boundary (see  \cite{CK07} and \cite{ClercMaslov2} and the references therein).

It takes values in $\{-n,\ldots, n\}$ and a triple $(Y_1,Y_2,Y_3)$ is \emph{maximal} if $\beta(Y_1,Y_2,Y_3)=n$.
\begin{example}\label{ExMaslovSp2n}
We will now express the Maslov index of a pairwise transverse triple $(Y_1,Y_2,Y_3)$ in the boundary $V=\Sym_n(\R)\subset \bar T_\Omega$.

Assume that there exists $g\in\Sp(2n,\R)$ such that $g(Y_1,Y_2,Y_3)=(-e,-i\varepsilon_k,e)$. The Cayley transform $c$ maps this $-e$ to $0\in V$ and $-iI_k$ to $I_k$. As mentioned above $c(I)$ is not defined. However we can add a point $\infty$ to $V$ whose stabilizer in $\Sp(2n,\R)$ is the stabilizer of $I\in \check S$. 

Let $Y_1=0$ and $Y_3=\infty$ and $Y_2$ be transverse to them. Since $Y_2$ is symmetric, there exists $M\in \GL(n,\R)$ and $k$ such that $Y_2=MI_k M^\top$. Observe that for 
\[
  g=\left(
    \begin{array}{cc}
  M^{-1}& \\ & M^\top
  \end{array}
\right)\in \Sp(2n,\R),
\]
we have $g(Y_1,Y_2,Y_3)=(0,I_k,\infty)$ and
\[
  \beta(Y_1,Y_2,Y_3)=\beta(0,I_k,\infty)=2k-n=\sgn I_k.
\]
\end{example}

 \begin{definition}\label{DefStructureGroup}
  Let $V$ be a Jordan algebra and $\tau(x,y):=\tr L(x\bar y)$. Then $V$ is \emph{semi-simple}, if $\tau$ is non-degenerate. The structure group $\Str(V)$ of  a semi-simple Jordan algebra $V$ is:
  \[
    \Str(V):=\{g\in \GL(V)|P(gx)=gP(x)g^*\},
  \]
  where $g^*$ is the adjoint of $g$ w.r.t. $\tau$. 
\end{definition}
Finally we introduce the character $\chi$ (cf. \cite[p.99]{CK07}):
\begin{definition}\label{DefChi}
  Let $V$ be a real or complex Jordan algebra and $\Str(V)$ be the structure group of $V$. Then we define the character $\chi$ on $\Str(V)$ via:
  \[
    {\det}_V(gx)=\chi(g){\det}_V(x),
  \]
  where $g\in \Str(V)$, $x\in V$ and $\det_V$ is the Jordan algebra determinant.
\end{definition}
\begin{lemma}
  Let $\Dca\subset V^\C$ be a circled bounded symmetric domain of tube type. Then $K$, the stabilizer of $0$ in the isometry group, is contained in the $\Str(V^\C)$
\end{lemma}
\begin{proof}
  From Lemma A.2 \cite{HS09}  follows that every $k\in K$ acts linearly on $V^\C$. Furthermore $k$ is an isometry, hence it preserves the Shilov boundary. Therefore, by Proposition X.3.1 in \cite{FK}, it is contained in $\Str(V^\C)$.
\end{proof}

\section[Formula]{A Formula for the Toledo Invariant}\label{SecFormula}
In this section we prove Theorem \ref{ThmFormulaIntro}. The proof is based on the expression of the Toledo invariant (as defined in the introduction) as the sum of generalized rotation numbers in \cite[Theorem 13]{Surface}. We will write the rotation number in terms of the Souriau index, which provides a link between this rotation number and the Maslov index. Throughout this section we denote by $G$ a Hermitian Lie group of tube type and by $\tilde G$ its universal cover.

\subsection{The Toledo Invariant}\label{SecToledo}
  
Let $g$ be the metric of $\Dca$ scaled such that the minimal sectional curvature is $-1$ and $J$ an invariant complex structure on $\Dca$, then 
\[
  \omega(X,Y):=g(JX,Y)
\]  
is a K\"ahler form on $\Dca$ (\cite[Lemma 2.1]{Anosov}). For $z_1,z_2,z_3\in \Dca$ we denote 
by $T(z_1,z_2,z_3)$ the geodesic triangle with vertices $z_1$, $z_2$ and $z_3$. Then
\[
  c(z_1,z_2,z_3):=\int_{T(z_1,z_2,z_3)}\omega.
\]
Now choose a base point $o\in \Dca$ and define for $g_1,g_2,g_3\in \Sp(2n,\R)$
\[
  c_G(g_1,g_2,g_3):=\frac{1}{2\pi} c(g_1o,g_2o,g_3o).
\]
This is a $G$-invariant homogeneous cocycle which is bounded (\cite{CO}).

 \begin{definition}\label{DefKaehler}
The cocycle $\c_G$ defines cohomology classes $\kappa_G\in \h^2_c(G,\R)$ resp. $\kappa_G^b\in \h^2_{cb}(G,\R)$, the \emph{K\"ahler class} resp. the \emph{bounded K\"ahler class}.
\end{definition}

Another cocycle for $G$ can be defined using the Maslov index. Let $b\in \check S$ be an arbitrary base point. Then
\[
  c_{\beta,b}(g_1,g_2,g_3):=\beta(g_1b,g_2b,g_3b).
\]
Note that $c_{\beta,b}$ and $c_{\beta,b'}$ are maybe not equal, but they are cohomologous (\cite[Prop 4.3]{Herm}).

Again we have:
 \begin{definition}\label{DefMaslovCocyle}
The cocycle $c_{\beta,b}$ defines cohomology classes $\kappa_\beta\in \h^2_c(G,\Z)$ resp. $\kappa_\beta^b\in \h^2_{cb}(G,\Z)$.
\end{definition}

\begin{proposition}(\cite[Prop.4.3]{Herm})\label{PropMaslovKaehler}
  \[
    \kappa_\beta= 2 \kappa_G \text{ and }\kappa^b_\beta= 2 \kappa^b_G
  \]
\end{proposition}
Both classes can be used to define the Toledo invariant as in the introduction. 

For the proof of Formula \ref{Formula} the following theorem is crucial. Burger, Iozzi and Wienhard define in \cite[Ch.7]{Surface} a generalized rotation number $\Rot_\kappa:G\rightarrow \R/\Z$ for every cohomology class $\kappa\in \h^2(G,\Z)$. It admits a unique lift $\Rott_\kappa:\tilde G\rightarrow \R$ with $\Rott_\kappa(e)=0$. 

\begin{theorem}\cite[Thm. 8.2]{Surface}\label{ThmToledo}
  Let $m\geq 1$ and $\varrho:\Gamma_{g,m}\rightarrow G$ be a maximal representation. Then
  \begin{equation}\label{ToledoRot}
    T_\varrho:=-\sum_{i=1}^m \Rott_{\kappa_G^b}(\tilde \varrho(C_i)),
  \end{equation}
  where $\tilde \varrho$ is some lift of $\varrho$ to $\tilde G$.
\end{theorem}
For the sake of completeness we state the analogous theorem for closed surfaces. The commutator map $G\times G\rightarrow \tilde G$ is defined as
\[
  [g,h]^\sim:=[\tilde g,\tilde h],
\]
where $\tilde g$ and $\tilde h$ are arbitrary lifts. 
\begin{theorem}\label{ThmRott2}
Let $\kappa\in H^2_{c}(G,\Z)$ and $\varrho:\Gamma_g\rightarrow G$ be a representation. Then
 \[
   T_{\varrho}=-\Rott_{\kappa_G^b} \left(\prod_{j=1}^g [\varrho(A_i),\varrho(B_i)]^\sim\right).
 \]
\end{theorem}
This is Theorem 8.3 in \cite{Surface}.

In the following Theorem we collect some properties of the Toledo invariant used later in the text. For more details see \cite{Surface}.
\begin{theorem}\label{PropPropToledo} (\cite[Thm 1, Prop. 3.2]{Surface})
  \begin{enumerate}
    \item $|T_\varrho|\leq |\chi(\Sigma)|\rk \Dca$ (Milnor-Wood inequality)
    \item $T_\bullet$ is continuous 
    \item If $\partial \Sigma=\emptyset$, then the image of $T_\bullet$ is finite
    \item If $\partial \Sigma\neq \emptyset$, then $T_\bullet$ is surjective on the interval 
    \[
      \big[- |\chi(\Sigma)|\Dca, |\chi(\Sigma)|\rk \Dca  \big]
    \]
    \item Let $\Sigma$ be a surface divided by a separating loop $l$ into two subsurfaces $\Sigma_1$ and $\Sigma_2$. Denote by $\varrho_1$ and $\varrho_2$ the restrictions of $\varrho$ to $\Sigma_1$ resp. $\Sigma_2$. Then 
    \[
      T_\varrho=T_{\varrho_1}+T_{\varrho_2}
    \] 
   \item  Let $\Sigma'$ be a surface obtained by cutting a surface $\Sigma$ along a non-separating loop. Let $i:\Sigma'\rightarrow \Sigma$ be the canonical map. Then
    \[
      T_{i^*\varrho}=T_\varrho.
    \]
  \end{enumerate}
\end{theorem}

In the introduction we defined maximal representations using the Toledo invariant. An equivalent characterization of maximal representations is \cite[Thm. 8]{Surface}: 
\begin{theorem}\label{ThmLimitCurve}
Fix $h$ a hyperbolization for $\Sigma$. 
A representation $\varrho:\Gamma \rightarrow G$ is maximal if and only if there exists a limit curve $\varphi: S^1\rightarrow \check S$ which is left continuous, $\varrho$-equivariant and which maps maximal triples in $S^1$ to maximal triples in $\check S$.
\end{theorem}


We end this section with a remark on the relation between homogeneous and inhomogeneous cocycles, because both types will appear in the following sections:
\begin{remark}\label{RemHomInhom}
  Group cohomology can be defined via homogeneous cocycles with boundary operator 
\[
  (\delta_n f)(g_0,\ldots, g_n):=\sum_{i=0}^n (-1)^i f(g_0,\ldots, \hat g_i,\ldots, g_n).
\]
or inhomogeneous cocycles with boundary operator
\begin{align*}
 (d_n f) (g_1,\ldots ,g_{n+1})=&\sum_{i=1}^n(-1)^i f(g_1,\ldots, g_ig_{i+1},\ldots,g_{n+1})\\ &+(-1)^{n+1}f(g_1,\ldots,g_n)+f(g_2,\ldots,g_{n+1}).
\end{align*}
  
They are intertwined as follows: let $f$ be an inhomogeneous $n$-cocycle, i.e. $\delta f=0$, then 
\[
  \tilde f(g_0,g_1,\ldots ,g_n):=f(g_0^{-1}g_1,g_1^{-1}g_2,\ldots,g_{n-1}^{-1}g_n)
\]
is an homogeneous cocycle. Its inverse is provided by
\[
  \bar h(g_1,\ldots, g_n):= h(e,g_1,g_1g_2, g_1g_2g_3,\ldots ,g_1g_2g_3\ldots g_n),
\]
where $h$ is a homogeneous cocycle.
\end{remark}

\subsection{The Souriau Index}\label{SecSouriau} In this section we define the Souriou index which is an essential building block in the proof of Theorem \ref{ThmFormulaIntro}. It is based on \cite{CK07}. Throughout this section $G$ is a Hermitian Lie group of tube type, $\Dca$ its centered bounded symmetric domain and we denote by $K$ the stabilizer of $0\in \Dca$. It is a maximal compact subgroup of $G$.

The universal covering of the Shilov boundary $\check S$, denoted by $\check R$, is given by 
\[
 \check R= \{ (\sigma,\theta)| \sigma \in \check S,\theta \in \R, {\det}_V\sigma=e^{ir\theta}\},
\]
where $r=\rk G$ (\cite[Thm 3.5]{CK07}) and ${\det}_V$ the Jordan algebra determinant. Denote by $\tilde G$ the universal cover of $G$.

\begin{definition}\label{defsouriau}
 Let $\tilde \sigma_1=(\sigma_1,\theta_1),\tilde \sigma_2= (\sigma_2,\theta_2)\in \check R$. They are \emph{transversal} ($\tilde \sigma_1\pitchfork \tilde \sigma_2$) if $\sigma_1$ and $\sigma_2$ are transversal. If $\tilde \sigma_1\pitchfork \tilde \sigma_2$, then there exists $\tilde g\in \tilde G$ such that 
\[
  \tilde g\tilde \sigma_1=\left(\sum e^{i\theta_j}c_j,\theta \right),\quad \tilde g\tilde \sigma_2=\left(\sum e^{i\phi_j}c_j,\phi \right)
\]
for some Jordan frame $\{c_i\}$.

Define the \emph{Souriau index} for transversal points:
\[
 m(\tilde \sigma_1,\tilde \sigma_2)= \frac{1}{\pi}\left[\sum\{\theta_j-\phi_j+\pi\}-r(\theta-\phi)\right],
\]
where $\{x\}$ is the unique representative of $[x]\mod 2\pi$ in $(-\pi,\pi)$. Note that transversality is equivalent to $\varphi_i \neq \theta_i$ for all $i$, whence $m$ is well-defined.\\

If $\tilde \sigma_1$ and $\tilde \sigma_2$ are not transversal then define
\[
  m(\tilde \sigma_1,\tilde \sigma_2):=\beta(\sigma_1,\sigma_2,\sigma_3)+m(\tilde \sigma_1,\tilde \sigma_3)+m(\tilde \sigma_3,\tilde \sigma_2),
\]
where $\tilde \sigma_3:=(\sigma_3,\theta_3)\in\check R$ is transversal to $\tilde \sigma_1$ and $\tilde \sigma_2$ and $\beta$ is the Maslov index.
\end{definition}

The following proposition collects properties of the Souriau index (\cite[Prop. 5.3 and 5.4]{CK07}):
\begin{proposition}\label{PropSouriau}
  The Souriau index is skew-symmetric, i.e. $m(\tilde \sigma_1,\tilde \sigma_2)=-m(\tilde \sigma_2,\tilde \sigma_1)$ and $\tilde G$-invariant.
\end{proposition}

\begin{remark}
  By  \cite[Thm. 6.1]{CK07} we have the following relation between the Souriau-index and the Maslov index on the Shilov boundary: let $a,b,c \in \check S$ and $\tilde a,\tilde b,\tilde c\in \check R$ be arbitrary lifts. Then
  \begin{equation}\label{FMaslovSouriau}
    \beta(a,b,c)=m(\tilde a,\tilde b)+m(\tilde b,\tilde c)+m(\tilde c,\tilde a).
  \end{equation}
\end{remark}

\begin{proposition}\label{LemSouriau}
  Let $\tilde x_1,\tilde x_2\in \check R$ two lifts of  $x\in \check S$ and $\tilde y\in \check R$ arbitrary. Then 
	\[
	  m(\tilde x_1,\tilde x_2)=m(\tilde x_1,\tilde y)+m(\tilde y,\tilde x_2).
	\]
\end{proposition}
\begin{proof}
  We use Formula \ref{FMaslovSouriau}.  We have:
  \[
    0=\beta(x,x,y)=m(\tilde x_1,\tilde x_2)+m(\tilde x_2,\tilde y)+m(\tilde y,\tilde x_1)
  \]
  and the statement follows from the skew-symmetry of $m$. 
\end{proof}
Immediate consequences from Proposition \ref{LemSouriau} and the $G$-invariance are:
\begin{lemma}\label{RotSouriau}
 Let $g\in G$ and $x\in \check S$ fixed by $g$. Let $\tilde g\in \tilde G$ and $\tilde x\in \check R$ be lifts. Then $m(\tilde g^n \tilde x,\tilde x)=n\cdot m(\tilde g\tilde x,\tilde x)$.
\end{lemma}

\begin{lemma}
  Let $x\in \check S$ and $\tilde x\in \check R$ a lift. Let $H<\tilde G$ be the lift of the stabilizer  of $x$ in $G$. Then 
  \[
     m(~\cdot~ \tilde x,\tilde x):H\rightarrow \R
  \]
  is a homogeneous quasimorphism.
\end{lemma}

\begin{lemma}\label{contrmaximal}
Let $g\in G$, $y\in \check S$ a fixed point of $g$ and $x\in \check S$ an arbitrary point. 
Let $\tilde x$ and $\tilde y$ be arbitrary lifts of $x$ and $y$ and $\tilde g$ a lift of $g$ which fixes $\tilde y\in \check R$.

Then
\[
  \beta(y,gx,x)=m(\tilde g\tilde x,\tilde x).
\]
\end{lemma}
\begin{proof}
  By Formula \eqref{FMaslovSouriau} and the assumptions we have
  \[
    \beta(y,gx,x)=m(\tilde y,\tilde g\cdot \tilde x)+m(\tilde g \cdot \tilde x,\tilde x)+m(\tilde x,\tilde y).
  \]
By $\tilde G$-invariance and skew-symmetry we have $m(\tilde y,\tilde g \tilde x)+m(\tilde x,\tilde y)=0$.
\end{proof}

\begin{remark}\label{RemPsi}
  Let $\sigma_1,\sigma_2\in \check S$ transversal. Let $k\in K$ such that
  \[
 k \sigma_1=\sum e^{i\theta_j}c_j,\quad k \sigma_2=\sum e^{i\phi_j}c_j,
  \]
  for some Jordan frame $\{c_j\}$.
  As in \cite[Ch.5]{CK07} we define 
  \[
    \Psi(\sigma_1,\sigma_2):=  \sum\{\theta_j-\phi_j+\pi\}.
  \]
  If $\sigma_1$ and $\sigma_2$ are not transversal, there exists $\sigma_3\in \check S$ transversal to both of them and we define
  \[
    \Psi(\sigma_1,\sigma_2):=\pi \beta(\sigma_1,\sigma_2,\sigma_3)+\Psi(\sigma_3,\sigma_2)+\Psi(\sigma_1,\sigma_3).
  \]
  In particular
  \begin{equation}\label{EqBetaPsi} \beta(\sigma_1,\sigma_2,\sigma_3)=\frac{1}{\pi}\big[\Psi(\sigma_1,\sigma_2)+\Psi(\sigma_2,\sigma_3)+\Psi(\sigma_3,\sigma_2)\big].
  \end{equation}
  $\Psi$ is invariant under $K$ and skew-symmetric (\cite[Prop.5.4]{CK07}).
\end{remark}
An important property of $\Psi$ is \cite[Formula (16)]{CK07}:
\[ 
  e^{2i\Psi(\sigma,\tau)}=({\det}_V \sigma)^2({\det}_V \tau)^{-2},
\]
where, again, ${\det}_V $ is the Jordan algebra determinant.

\begin{proposition}\label{PropPsiHom}
  Fix $b\in \check S$. Then the map
  \[
    f:\begin{cases}
      K \rightarrow \R/\Z\\
      k\mapsto \left[\frac{1}{\pi}\Psi(b,kb)\right]
    \end{cases}
  \]
  is a homomorphism. It does not depend on $b$.
\end{proposition}

\begin{proof}
  Let $k_1,k_2\in K$. Then 
  \begin{align*}
     &e^{2i\Psi(b, k_1k_2 b)}=({\det}_V b)^2({\det}_V k_1k_2 b)^{-2}=\chi(k_1k_2)^{-2} ({\det}_V b)^2 ({\det}_V b)^{-2}\\
     =&\chi(k_1)^{-2}({\det}_V b)^{2} ({\det}_V b)^{-2} \chi(k_2)^{-2}({\det}_V b)^2 ({\det}_V b)^{-2}=e^{2i\big(\Psi(b,k_1b)+\Psi(b,k_2 b)\big)},
  \end{align*}
where $\chi $ is the character on $\Str(V^\C)$ introduced in Definition \ref{DefChi}. Therefore $f$ is a homomorphism.\\
  Now we show independence of $b$: let $b,b'\in \check S$. Then there exists $l\in K$ such that $b'=lb$. Then $\Psi(b',kb')=\Psi(b,l^{-1}kl b)$ for all $k\in K$. Since $f$ is a homomorphism into the abelian group $\R/\Z$, $\Psi(b,l^{-1}klb )=\Psi(b,kb)$, for all $k$. Hence $f$ does not depend on $b$.
\end{proof}

\subsection{The Rotation Number and the Souriau Index}\label{SecRot}
Recall $\kappa_\beta\in \h^2(G,\Z)$, defined in Section \ref{SecToledo}. We will express the rotation number $\Rot_{\kappa_\beta}$ in terms of  the map $\Psi$ defined in Section \ref{SecSouriau}. 
\begin{proposition}\label{PropRotPsi}
  Fix $b\in \check S$. Given $g\in G$, let $g=g_eg_hg_u$ its refined Jordan decomposition. Let $k\in C(g_e)\cap K$, where $C(g_e)$ is the conjugacy class of $g_e$. Then
   \[
     \Rot_{\kappa_\beta}(g)=\left[\frac{1}{\pi}\Psi(b,kb)\right] \in \R/\Z.
   \]
\end{proposition}
For the refined Jordan decomposition see \cite{Kostant} or \cite{BorelJordanDecomp}. 
\begin{proof}
  Recall that $\frac{1}{\pi}\Psi(b,kb)$ defines a homomorphism $K \rightarrow \R/\Z$ (Proposition \ref{PropPsiHom}). Define $B:=\overline {\langle g\rangle} $. Given $h,h'\in B$ the refined Jordan decompositions are compatible, i.e. $h'_e,h'_u,h'_h,h_u,h_e,h_h$ commute pairwise. Since $\Rot$ is conjugation invariant (\cite[Lem.7.2]{Surface}), we can assume that the elliptical part $k$ of $g$ in the refined Jordan decomposition is in $K$. Under this assumptions this holds for all elements in $B$, because the Jordan decompositions are compatible.\\
  We are searching for $f_B:B\rightarrow \R$ which defines a homomorphism $B\rightarrow \R/\Z$ and such that $\partial f_b$ is a representative of $\kappa_\beta$. Let $g=g_ug_hk$ be the refined Jordan decomposition of $g$. Choose $b\in \check S$ such that $g_ug_h$ fixes $b$. Then for any $h\in B$, $h_uh_h$ fixes $b$, since $h$ is the limit of powers of $g$ and their refined Jordan decompositions are compatible. Now we define
  \[
    f_B:
    \begin{cases}
      B\rightarrow \R\\
      h \mapsto \frac{1}{\pi} \Psi(b,hb). 
    \end{cases}
  \]
 Note that $\Psi(b,kb)=\Psi(b,h_eb)$ and therefore by Proposition \ref{PropPsiHom} it defines a homomorphism $B\rightarrow \R/\Z$.\\
 It remains to show $\delta f_B=\kappa_\beta|_B$.
  Let $g_1,g_2\in B$ and denote by $k_1,k_2\in K$ the respective elliptic parts of the refined Jordan decomposition:
  \begin{align*}
    \delta f_B(g_1,g_2)=&f_B(g_1)-f_B(g_1g_2)+f_B(g_1)\\= &\frac{1}{\pi}\left[\Psi(b,k_1b)-\Psi(b,k_1k_2b)+\Psi(b,k_2b)\right]\\
    =&\frac{1}{\pi}\left[ \Psi(b,k_1b )+\Psi(k_1 b,k_1k_2 b)+\Psi(k_1 k_2 b, b)\right]\\  
 =&\beta(b, k_1 b,k_1 k_2 b) =\beta(b,g_1 b,g_1 g_2b),
  \end{align*}
  where we used $K$-invariance of $\Psi$ and Formula \eqref{EqBetaPsi}. 
  This finishes the proof, because $\beta(b,g_1b,g_1g_2b)$ is an inhomogeneous cocycle defining $\kappa_\beta$ (c.f. Remark \ref{RemHomInhom}).
\end{proof}

\begin{definition}
  Let $G$ be a group and $f:G\rightarrow \R$ be a map. Then $f$ is a \emph{quasimorphism}, if there exists $C\in \R$ such that
  \[
    | f(gh)-f(g)-f(h)|\leq C, \quad \forall g,h\in G.
  \]
  A quasimorphism is \emph{homogeneous}, if $f(g^n)=nf(g)$ for all $g\in G$ and $n\in \Z$.
\end{definition}

\begin{proposition}\label{PropQM}
   Let 
     \[ 
    \tau(\tilde g):=\lim_{n\rightarrow \infty}\frac{m(\tilde b,\tilde g^n\tilde b)}{n}.
  \]

    be the homogenization of $\tilde g\mapsto m(\tilde b,\tilde g \tilde b)$. Then $\tau$ is a quasimorphism of $\tilde G$ and
   \[  
     \tau=-\widetilde{\Rot}_{\kappa_\beta}
   \]
\end{proposition}
\begin{proof}
  From \cite[Ch.10]{CK07} we get that $\tau$ does not depend on $\tilde b$ and it is a homogeneous quasimorphism.
  Since $\tau(e)=\widetilde \Rot_{\kappa_\beta}(e)=0$ it is enough to show that 
  \[ 
    \tau(\tilde g)= -\Rot_{\kappa_\beta}(g)~\mod \Z
  \]
  for $g\in G$ and $\tilde g$ any lift, because the lifts of both sides are unique. First note that both sides are conjugation invariant and only depend on the elliptic part of the refined Jordan decomposition of $g$. Hence it is enough to show this equality for $g=k\in K$. \\
 By Proposition 10.4 in \cite{CK07} and Proposition \ref{PropRotPsi} we have 
 \[
   e^{2i\pi \tau(\tilde k)}=\chi(k)^2=e^{-2i \pi\Psi(b,kb)/\pi}=e^{-2i\pi \Rot_{\kappa_\beta}} 
 \]
 for all $k\in K$ and all lifts $\tilde k$ of $k$.
\end{proof}

\begin{corollary}\label{CorRotTau}
   Let $g\in G$ and $x\in \check S$ a fixed point of $g$. Let $\tilde g\in \tilde G$ and $\tilde x\in \check R$ be lifts. Then
   \[
     \widetilde{\Rot}_{\kappa_\beta}(\tilde g)=m(\tilde g \tilde x, \tilde x)
   \]
   and $\Rott_{\kappa_\beta} (\tilde g)=0$ if $\tilde g$ has a fixed point in $\check R$.
\end{corollary}
\begin{proof}
  We calculate
  \[
    \tau(\tilde g)=\lim_{n\rightarrow \infty}\frac{m(\tilde g^n\tilde x,\tilde x)}{n}=\lim_{n\rightarrow \infty}\frac{n\cdot m(\tilde g\tilde x,\tilde x)}{n}=m(\tilde g \tilde x,\tilde x),
  \]
  where we have used Lemma \ref{RotSouriau}.
\end{proof}

Recall the classical \emph{translation number} $T$ on $\widetilde{\PSL(2,\R)}$. Let $g\in \PSL(2,\R)$ and $\tilde g\in \widetilde{\PSL(2,\R)}$ a lift and $\tilde x \in \R$. Then:
\[
  T(\tilde g)=\lim_{n\rightarrow \infty}\frac{\tilde g^n\tilde x-\tilde x}{n}.
\]
\begin{remark}\label{RemTransl}
The number $T(\tilde g)$ is independent of $\tilde x$. If $g\in \PSL(2,\R)$ has  a fixed point $x\in S^1$ and $\tilde g$ and $\tilde x$ are lifts, then $T (\tilde g)=\tilde g\tilde x-\tilde x$.
\end{remark}
\begin{proposition} \label{PropRotTransl}
  Let $g\in \PSL(2,\R)$ and $\tilde g\in \widetilde{\PSL(2,\R)}$ a lift. 
  \[
    \Rott_{\kappa_\beta}(\tilde g)=-T(\tilde g).
  \]
\end{proposition}
\begin{proof}
  The map $T$ satisfies $T(e)=0$ and 
  \[
    -T(\tilde g)=\Rot_{\kappa_\beta} \mod \Z
  \]
  (see Proposition \ref{PropRotPsi} and Remark \ref{RemPsi}). Hence $-T(\tilde g)=\Rott_{\kappa_\beta}(\tilde g)$.
\end{proof}

We finish this subsection with a short lemma, which we will need later.

\begin{lemma}\label{homquasim}
  Let $f$ be a homogeneous quasimorphism. Then $f(g)=-f(g^{-1})$. 
\end{lemma}
\begin{proof}
  Since $f(e)=f(e^n)=nf(e)$ for all $n\in \N$, $f(e)=0$. By the definition of quasimorphism we have:
  \[
	n|f(g)+f(g^{-1})|=|f(g^n)+f(g^{-n})-f(e)|\leq r
  \]
  for all $n\in \N$, hence $f(g)+f(g^{-1})=0$.
\end{proof}

\subsection{Proof of Theorem \ref{ThmFormulaIntro}}
\begin{proof}[Proof of Theorem \ref{ThmFormulaIntro}]
For the calculation of the Toledo invariant according to Theorem \ref{ThmToledo} we have to choose a lift $\tilde\varrho$. Define $c_i:=\varrho(C_i)$. We choose $\tilde c_1:=\tilde \varrho(C_1)$ and $\tilde c_2:=\tilde \varrho(C_2)$ such that both have fixed points in $\check R$, which is possible since $c_1$ and $c_2$ have fixed points in $\check S$. Then $\tilde c_3:=(\tilde c_2 \tilde c_1)^{-1}$ is a lift of $c_3$ and by definition $\tilde c_1$, $\tilde c_2$ and $\tilde c_3$ define a representation of $\Gamma_{0,3}$ into $\tilde G$. 

Since $\tilde c_1$ and $\tilde c_2$ have fixed points in $\check R$, we have $\Rott_{\kappa_\beta}(\tilde c_1)=\Rott_{\kappa_\beta}(\tilde c_2)=0$ (Corollary \ref{CorRotTau}). Therefore it suffices to calculate $\Rott_{\kappa_\beta}(\tilde c_3)$.  Since $\Rott$ is a homogeneous quasimorphism (\cite[Thm.11]{Surface}), we have $\Rott_{\kappa_\beta}(\tilde c)=-\Rott_{\kappa_\beta}(\tilde c^{-1})$ (see \cite{Surface} and Prop \ref{homquasim}), hence 
\[
  \Rott_{\kappa_\beta}(\tilde c_3)=\Rott_{\kappa_\beta}((\tilde c_2\tilde c_1)^{-1})=-\Rott_{\kappa_\beta}(\tilde c_2\tilde c_1).
\]
By Corollary \ref{CorRotTau}, $\Rott_{\kappa_\beta}(\tilde c_2 \tilde c_1)=m(\tilde y_3, \tilde c_2 \tilde c_1 \tilde y_3)$, since $\tilde y_3$ is the lift of a fixed point of $c_2c_1$. Furthermore (Lemma \ref{LemSouriau})
\[
  m(\tilde c_2 \tilde c_1 \tilde y_3,\tilde y_3)=m(\tilde c_1\tilde y_3, \tilde y_3)+m(\tilde c_2\tilde c_1 \tilde y_3, \tilde c_1 \tilde y_3)
\]
We can summarize the discussion above to
\begin{align*}
  T_\varrho=& -\Rott_{\kappa_\beta}(\tilde c_1)-\Rott_{\kappa_\beta}(\tilde c_2)-\Rott_{\kappa_\beta}(\tilde c_3)=-\Rott_{\kappa_\beta}(\tilde c_3)\\ =&\Rott_{\kappa_\beta}(\tilde c_2\tilde c_1)
  =m(\tilde c_2 \tilde c_1\tilde y_3, \tilde y_3)=m( \tilde c_1\tilde y_3, \tilde y_3)+m( \tilde c_2 (\tilde c_1) \tilde y_3,\tilde c_1\tilde y_3)
\end{align*}

Now we will express the right hand side in terms of the Maslov index. By Lemma \ref{contrmaximal} we get
  \[
     m( \tilde c_1\tilde y_3,\tilde y_3)+m(\tilde c_2( \tilde c_1 \tilde y_3),\tilde c_1 \tilde y_3)=\beta(y_1,c_1\cdot y_3,y_3)+\beta(y_2,c_2(c_1 y_3),c_1 y_3)
  \]
 and
  \begin{align*}
   &\beta(y_2,c_2(c_1 y_3),c_1 y_3)+\beta(y_1,c_1\cdot y_3,y_3)\\
    =&\beta(c_1\cdot y_3,y_2,y_3)+\beta(y_1,c_1\cdot y_3,y_3)\\
    =& \beta(y_1,c_1\cdot y_3,y_2)+\beta(y_1,y_2,y_3).
  \end{align*}
  In the first step we used $G$-invariance, the fact that $c_2c_1=c_3^{-1}$ and the anti-symmetry and in the second step the cocycle property. This finishes the proof.
\end{proof}

\begin{remark}
  The proof of Theorem \ref{ThmFormulaIntro} relies on the fact that $m( \tilde x,\tilde g \tilde x)$ behaves like a homogeneous quasimorphism if $\tilde x$ is the lift of a fixed point of $g$. Therefore we cannot expect a similar formula for representations where one or more generators do not have a fixed point. In particular the right hand side of the formula only takes a finite number of values, but since $\Gamma_{0,3}$ is free, the Toledo invariant is surjective on the interval $[- r,r]$, where $r$ is the rank of $G$ (\cite[Thm.1]{Surface}).
\end{remark}

\begin{corollary}\label{boundarycond}
  Let $\varrho$ be a maximal representation of $\Gamma_{0,3}$ into a Hermitian Lie group $G$ of tube type. Then  \begin{itemize}
	\item[(i)] each $c_i:=\varrho(C_i)$ has a fixed point $y_i\in \check S$,
	\item[(ii)] $\beta(y_1,y_2,y_3)=r$,
	\item[(iii)] $\beta(y_1,c_1\cdot y_3,y_2)=r$.
\end{itemize}
Conversely if $\varrho$ satisfies (i)-(iii), then $\varrho$ is maximal.
\end{corollary}

\begin{proof} Let $\varrho:\Gamma_{0,3}\rightarrow G$ be a maximal representation. 
Then (i) follows from \cite[Lemma 8.8]{Surface}.
Properties (ii) and (iii) as well as the converse follow immediately from Formula \eqref{Formula}.
\end{proof}

\begin{remark}\label{RemFPTransversal}
  An important consequence of Proposition \ref{boundarycond} (ii) is the fact that given a fixed point $y_i$ of $c_i$, all fixed points of $c_j$ with $j\neq i$ are transverse to $y_i$. In particular, if we calculate in $T_\Omega$ as in Section \ref{SecBSD}, we can assume that $y_i=\infty$. Then every fixed point of $c_j$, $i\neq j$ is contained in $V$. 
\end{remark}

\begin{remark}
Note that there there are other possibilities to state the formula:

  The Maslov index is antisymmetric and $G$-invariant, hence
  \begin{equation}\label{EqMaslovBoundary}
    \beta(y_1,c_1\cdot y_3,y_2)=\beta(y_2,c_2\cdot y_1,y_3)=\beta(y_3,c_3\cdot y_2,y_1),
  \end{equation}
  i.e. we can also express the Toledo invariant in terms of $c_2y_1$ or $c_3y_2$ and the fixed points $y_1$, $y_2$ and $y_3$.
\end{remark}

\section{Representations of $\Gamma_{0,3}$ into $\Sp(2n,\R)$}\label{Sec03}
The purpose of this section is the proof of Theorem \ref{MainThmIntro} using Proposition \ref{Prop03} and Proposition \ref{PropNormalForm} below. In the following remark we will fix the model for the symmetric space and its boundary on which we perform all calculations.

\begin{remark}\label{RemAddOnePoint}
Let $\varrho$ be a maximal representation into $\Sp(2n,\R)$. Then the image of its limit  curve in $\check S$ is a $\varrho(\Gamma)$-invariant subset of $\check S$ such that any two different points are transversal. 
Recall from Remark \ref{RemVopendense} that the image of $V=\Sym_n(\R)\subset \bar T_\Omega$ under the Cayley transform consists of all elements which are transversal to the point $I\in \check S$. Hence we can assume that $\varphi(S^1)-\{\pt\}\subset p(\Sym_n(\R))\subset \check S$. So if we add a point $\infty$ to $\bar T_\Omega$ whose stabilizer in $\Sp(2n,\R)$ is the same as the one of $I\in \check S$, then $\Sym_n(\R)\cup \{\infty\}$ contains a $\varrho(\Gamma)$-invariant subset. Moreover if $\infty$ is a fixed point of, say, $\varrho(C_3)$ then all fixed points of $\varrho(C_1)$ and $\varrho(C_2)$ are contained in $\Sym_n(\R)$, even the ones that are not on the limit curve (c.f. Remark \ref{RemFPTransversal}). So we add $\infty$ to $\bar T_\Omega$. This point ``equals $p(I)$'' and it can be seen as the limit point of all sequences of diagonal matrices in $\bar T_\Omega$ where \emph{all} eigenvalues go to infinity. Using this description we get the following action of $\Sp(2n,\R)$ on $\infty$:
\[
  \left(\begin{array}{cc}
	A&B\\
	C &D
\end{array}\right)\infty =AC^{-1}
\] 
and 
\[
  \left(\begin{array}{cc}
	A&B\\
	C &D
\end{array}\right)X=\infty
\]
if and only if $CX+D=0$.
\end{remark}
Recall that for $k\in\{0,\ldots,n\}$, we defined
\[
  I_k:=\left(\begin{array}{cc}
	1_k&\\
	 &-1_{n-k}
\end{array}\right),
\]
where $1_k$ denotes the $k\times k$ identity matrix.

 \begin{proposition}\label{Prop03}
Let $\varrho:\Gamma_{0,3}\rightarrow \Sp(2n,\R)$ such that $c_1:=\varrho(C_1)$ fixes $0$, $c_2:=\varrho(C_2)$ fixes $I_i$, $\varrho(C_3)$ fixes $\infty$ and such that $\varrho(C_1)\infty$, $0$, $e$ and $\infty$ are pairwise transversal. Then there exists $X_1,X_2,X_3\in \GL(n,\R)$ with $X_3(X_2^\top)^{-1}X_1$ symmetric such that
  \begin{align*}
    c_1=&\left( \begin{array}{cc}X_1&0\\X_2^{-1}X_3^\top+I_iX_1&(X_1^\top)^{-1}\end{array}\right),
  \\
  c_2=&\left( \begin{array}{cc}-I_i(X_2^\top)^{-1}-I_iX_3^{-1}X_1^\top I_i-X_2 I_i & I_iX_3^{-1}X_1^\top+X_2 \\ -(X_2^\top)^{-1} -X_3^{-1}X_1^\top I_i & X_3^{-1}X_1^\top \end{array}\right)\\
  c_3=&\left( \begin{array}{cc}(X_3^\top)^{-1}&-(X_3^\top)^{-1}I_i-X_1^{-1}X_2^\top \\0&X_3\end{array}\right)
\end{align*}
and 
\begin{equation}\label{EqFormulaToledoSp2nR}
  T_\varrho=i+j-n,
\end{equation}
where $j=\frac{1}{2}\left(\beta(0,\varrho(C_1)\infty,I_i)+n\right)$.
\end{proposition}

This proposition  will be proven in Section \ref{SecParam} below.

\begin{remark}\label{RemNonUnique}
Proposition \ref{Prop03} is not enough to prove Theorem \ref{MainThmIntro}, because we want to parametrize conjugacy classes of representations and there are conjugate representations which both have the form as in Proposition \ref{Prop03} but for nonconjugate $(X_1,X_2,X_3)$. Indeed, let $\varrho$ be a maximal representation with $\varrho(C_i)=c_i$ as in Proposition \ref{Prop03} such that some $c_i$ has more than one fixed point in the Shilov boundary. Such $\varrho$ exists,as, for example, hyperbolic elements in $\SL(2,\R)$ have two fixed points in the Shilov boundary $S^1$. Take any triple $(Y_1,Y_2,Y_3)$ of fixed points for the $c_i$. By Corollary \ref{boundarycond} we have: $\beta(Y_1,Y_2,Y_3)=n$, hence there exists $g\in \Sp(2n,\R)$ such that $g(0,e,\infty)=(Y_1,Y_2,Y_3)$. Then $\varrho$ as well as $g\varrho g^{-1}$ are in the form as in Proposition \ref{Prop03} and hence they give different parameters $X_i$ for the same conjugacy class. To circumvent this problem we need canonical representatives for each conjugacy class. Proposition \ref{PropNormalForm} below yields a dynamical criterion to obtain a canonical fixed points for any $\varrho(C_i)$, map them to $(0,e,\infty)$ and apply Proposition \ref{Prop03} and obtain coordinates $(X_1,X_2,X_3)$ which are unique up simultaneously conjugation with the stabilizer of the triple $(0,e,\infty)$.
\end{remark}

Denote by $\sigma(A)$ the spectrum of $A\in \GL(n,\R)$.
\begin{proposition}\label{PropNormalForm}
  Let \begin{equation}
     c=
     \left(\begin{array}{cc}
	A&0\\
	A+(A^\top)^{-1}S &(A^\top)^{-1}
\end{array}\right)\in \Sp(2n,\R),
  \end{equation}
  where $A$ is invertible and $S$ symmetric and positive definite.
  \begin{enumerate}
    \item If $\sigma(A)\subset S^1$ , then $0$ is the unique fixed point of $c$ in $V$.
    \item If $\sigma(A)\nsubseteq S^1$, then $c$ has a \emph{unique} fixed point $Y$ in which the action is non-expanding. There exists $g\in \Sp(2n,\R)$ which maps $0$ to $Y$ such that $gcg^{-1}=\left(\begin{array}{cc}
	A'&0\\
	C' &(A'^\top)^{-1}
\end{array}\right)$ 
where $A'$ has  no eigenvalue of absolute strictly value bigger that $1$ and $C'$ is some $n\times n$-matrix.
  \end{enumerate}
  In case (i) we call 0,  in case (ii) we call $Y$  the canonical fixed point of $c$.
\end{proposition}

We prove Proposition \ref{PropNormalForm} in Section \ref{SecNormalForm}.

\begin{proof}[Proof of Theorem \ref{MainThmIntro}]
First we show that the map $f$ is well-defined. A direct calculation shows that for $c_1$, $c_2$ and $c_3$ the product $c_3c_2c_1$  is equal to the identity if and only if $X_3(X_2^\top)^{-1}X_1$ is symmetric. It remains to show that $\bar f(X_1,X_2,X_3)$ does not depend on the representative of the equivalence class in $R/\On(n)$. Let $(X_1,X_2,X_3)\in R$ and $k\in \On(n)$ and $c_1$, $c_2$ and $c_3$ be the generators of $\bar f(X_1,X_2,X_3)$.  Then the generators of $\bar f(kX_1k^{-1},kX_2k^{-1},kX_3k^{-1})$ are $lc_1l^{-1}$, $lc_2l^{-1}$ and $lc_3l^{-1}$ where 
\[
  l=\left(
\begin{array}{cc}
	k&0\\
	0& k
\end{array}\right),
\]
hence $\bar f(X_1,X_2,X_3)$ and $\bar f(kX_1k^{-1},kX_2k^{-1},kX_3k^{-1})$ are in the same conjugacy class in $\Rep(\Gamma_{0,3},\Sp(2n,\R))$ and $f$ is well-defined. Furthermore $\bar f(X_1,X_2,X_3)$ is maximal by Formula \eqref{EqFormulaToledoSp2nR} in Proposition \ref{Prop03}.

To show that $f$ is bijective, we construct an inverse map. The main ingredients are Proposition \ref{Prop03} and Proposition \ref{PropNormalForm}. Let $\varrho'$ be a maximal representation and define $c_i':=\varrho'(C_i)$.
  
Every $c_i'$ is conjugate to some $c$ as in Theorem \ref{PropGlueSp2nRIntro} (i). Therefore we can apply Proposition \ref{PropNormalForm} and each $c_i'$ has a canonical fixed point $X^+_i$ (in which it acts non-expandingly). Since the triple $(X_1^+,X_2^+,X_3^+)$ is maximal (Cor. \ref{boundarycond}), there exists $h\in \Sp(2n,\R)$ which maps this triple to $(0,e,\infty)$. Therefore the images of $C_1$, $C_2$ resp. $C_3$ under the representation $\varrho:=h\varrho'h^{-1}$  fix $0$, $e$ resp. $\infty$. We can apply Proposition \ref{Prop03} and get $(X_1,X_2,X_3)$ with $X_3(X_2^\top)^{-1}X_1$ symmetric positive definite, such that $\varrho = \bar f(X_1,X_2,X_3)$. By construction the $\varrho(C_i)$ are non-expanding in $0$, $e$ and $\infty$, respectively, hence $(X_1,X_2,X_3)\in R$.  Since $h$ is unique up to left-multiplication with $k\in \stab_G((0,e,\infty))=\On(n)$, the triple $(X_1,X_2,X_3)$ is unique up to conjugation by an element in $\On(n)$. This provides a map inverse to $\bar f$.
\end{proof}

\subsection{Parameters for Representations of $\Gamma_{0,3}$}\label{SecParam}
In this section we proof Proposition \ref{Prop03}. 

\begin{proof}[Proof of Proposition \ref{Prop03}] 
The proof relies on the information for the position of the points $c_1\cdot\infty$, $c_2\cdot 0$ and $c_3\cdot I_i$ in $V$ expressed in terms of the Maslov index in Corollary \ref{boundarycond}.

Define $Z_1$, $Z_2$ and $Z_3$ in $\check S$ by
\[
  Z_1:=c_3\cdot I_i(=c_1^{-1}\cdot I_i),\quad Z_2:=c_1\cdot \infty(=c_2^{-1}\cdot \infty),\quad Z_3:=c_2\cdot 0(=c_3^{-1}\cdot 0),
\]
where the second equalities hold because $c_3c_2c_1=\id$.

We can summarize that to the following conditions for the $c_i$:
\begin{equation}\label{EqMappings}
c_1:\begin{cases}
  0\mapsto 0\\
  Z_1\mapsto I_i\\
  \infty \mapsto Z_2
\end{cases} \quad
c_2:\begin{cases}
  I_i\mapsto I_i\\
  Z_2\mapsto \infty\\
  0\mapsto Z_3
\end{cases} \quad
c_3:\begin{cases}
  \infty \mapsto \infty\\
  Z_3\mapsto 0\\
  I_i \mapsto Z_1
\end{cases} 
\end{equation}
These conditions give a system of three equations for $c_1$, $c_2$ and $c_3$ which determine each of them up to an element of the stabilizer of a certain triple in $\check S$. We will determine all solutions of these equations. 

By \eqref{EqMaslovBoundary} we have
\begin{equation}
  \label{MaxTrip}\beta(0,Z_2,I_i)=\beta(Z_1,0,\infty)=\beta(I_i,Z_3,\infty)=2j-n
\end{equation}
for some $j\in \{0,\ldots,n\}$
and by Example \ref{ExMaslovSp2n} this is the case if and only if the symmetric matrices $-Z_1$, $Z_2^{-1}-I_i$ and $Z_3-I_i$ have the same signature.
If $Z$ is a symmetric matrix of signature $2j-n$, then there exists an invertible matrix $M$, such that $Z=MI_jM^\top$.

For such a $Z$, the equation $Z=MI_jM^\top$ determines $M$  uniquely up to a right multiple $k\in \On(j,n-j)$. Note that $k\in \On(j,n-j)$ if and only if $kI_l k^\top=I_l$ or equivalently $I_lkI_l=(k^\top)^{-1}$. So for the rest of the proof we fix $M_1$, $M_2$ and $M_3$ s.t.
\begin{equation}\label{Zi}
  -Z_1=M_1 I_j M_1^\top,~Z_2^{-1}-I_i=M_2I_j M_2^\top,~Z_3-I_i=M_3I_j M_3^\top,
\end{equation}

For any $k_1\in \On(j,n-j)$ we define:
\[
  c_1:=\left(\begin{array}{cc}
	(M_2^\top)^{-1}k_1 M_1^{-1}&0\\
	M_2I_jk_1 M_1^{-1}+ I_i(M_2^\top)^{-1}k_1M_1^{-1}& M_2 (k_1^\top)^{-1} M_1^\top
\end{array}\right),
\]
It is easy to check that it maps $(0,Z_1,\infty)$ to $(0,I_i,Z_2)$. Let $c\in \Sp(2n,\R)$ be another map which also maps the triple $(0,Z_1,\infty)$ to $(0,I_i,Z_2)$. We show now that $c=c_1$ for some $k_1\in  \On(j,n-j)$. Decompose $c_1$ as follows:
\[
  c_1=\underbrace{\left(
\begin{array}{cc}
	(M_2^\top)^{-1} &0\\
	M_2I_j+I_i(M_2^\top)^{-1} & M_2
\end{array}\right)}_{=:c'}\underbrace{\left(
\begin{array}{cc}
	k_1 &0\\
	0 & (k_1^\top)^{-1}
\end{array}\right)}_{=:l}
\underbrace{
\left(
\begin{array}{cc}
	M_1^{-1} &0\\
	0 & M_1^\top
\end{array}\right)}_{=:c''}.
\]
Then $c'':(0,Z_1,\infty)\mapsto (0,-I_j,\infty)$ and $c':(0,-I_j,\infty)\mapsto (0,I_i,Z_2)$. Hence $(c'')^{-1}c(c')^{-1}$ is an element of the stabilizer of $(0,-I_j,\infty)$ which is 
\[
  \left\{\left(\left.
\begin{array}{cc}
	k &0\\
	0 & (k^\top)^{-1}
\end{array}\right) \right| k\in \On(j,n-j) \right\}.
\]
Analogously we define for $k_2,k_3\in \On(p,q)$:
\[
c_2=\left(
\begin{array}{cc}
	A&B\\
	C&D
\end{array}\right)
\]
with
\begin{align*}
A=&-M_3I_jk_2M_2^{-1}I_i-I_i(M_3^\top)^{-1}k_2I_jM_2^\top -I_i(M_3^\top)^{-1}k_2 M_2^{-1}I_i\\
B=& M_3I_j k_2 M_2^{-1}+I_i(M_3^\top)^{-1}k_2M_2^{-1}\\
C=&	-(M_3^\top)^{-1}k_2 I_jM_2^\top -(M_3^\top)^{-1} k_2 M_2^{-1}I_i\\
D=& (M_3^\top)^{-1}k_2 M_2^{-1}
\end{align*}
and
\[
c_3=\left(
\begin{array}{cc}
	M_1 k_3 M_3^{-1}& -M_1 k_3 I_j M_3^\top -M_1 k_3M_3^{-1} I_i\\
	0& (M_1^\top)^{-1}(k_3^\top)^{-1} M_3^\top
\end{array}\right)
\]

A direct calculation (e.g. by calculating $c_1c_3$ and comparing the product with $c_2^{-1}$ using Formula \eqref{EqInverse}) shows that the product $c_3c_2c_1$ is the identity if and only if $k_3(k_2^\top)^{-1}k_1=1$. 

Note that the sixtuples $(M_1,M_2,M_3,k_1,k_2,k_3)$ as above define precisely the same representation as $(M_1k_1^{-1},M_2,M_3(k_2^\top)^{-1},e,e,e)$. Hence the $k_i$ do not give new parameters. We choose $k_1=k_3=I_j$ and $k_2=\id$ and get:

\begin{align*}
    c_1=&\left( \begin{array}{cc}(M_2^\top)^{-1}I_jM_1^{-1}&0\\M_2M_1^{-1}+I_i(M_2^\top)^{-1}I_jM_1^{-1}&M_2I_jM_1^\top\end{array}\right)\\
    c_2=&\left( \begin{array}{cc}A& B\\ C&D \end{array}\right)
\\
  c_3=&\left( \begin{array}{cc}M_1I_jM_3^{-1}&-M_1I_jM_3^{-1}I_i-M_1M_3^\top \\0&(M_1^\top)^{-1}I_jM_3^\top\end{array}\right)
\end{align*}
with
\begin{align*}
A=&-M_3I_jM_2^{-1}-I_i(M_3^\top)^{-1}I_jM_2^\top -I_i(M_3^\top)^{-1} M_2^{-1}I_i\\
B=& M_3I_j  M_2^{-1}+I_i(M_3^\top)^{-1}M_2^{-1}\\
C=&	-(M_3^\top)^{-1} I_jM_2^\top -(M_3^\top)^{-1}  M_2^{-1}I_i\\
D=& (M_3^\top)^{-1} M_2^{-1}
\end{align*}

Now define
\begin{align*}
X_1:=&(M_2^\top)^{-1}I_j M_1^{-1}\\
X_2:=&M_3I_jM_2^{-1}\\ 
 X_3:=&(M_1^\top)^{-1}I_j M_3^\top.
\end{align*}
Observe $X_3(X_2^\top)^{-1}X_1=(M_1^{-1})^\top I_j M_1^{-1}$ is a symmetric matrix of signature $j$.

We use Formula \eqref{Formula} to calculate the Toledo invariant. By construction the generators have $y_1=0$, $y_2=I_i$ and $y_3=\infty$ as fixed points, hence 
$  \beta(y_1,y_2,y_3)=2i-n$. For the second term in Formula \eqref{Formula} we calculate 
\begin{equation}\label{EqC3e}
  \beta(y_3,c_3\cdot y_2,y_1)=\beta(c_3\cdot y_2,y_1,y_3)=\beta(c_3\cdot I_i,0,\infty).
\end{equation}
By Example \ref{ExMaslovSp2n}, $-\beta(c_3\cdot I_i,0,\infty)$ is the signature of the symmetric matrix $c_3\cdot I_i$ and we calculate
\[
c_3\cdot I_i=(X_3^\top)^{-1}I_iX_3^{-1}-(X_3^\top)^{-1}I_iX_3^{-1}-X_1^{-1}X_2^\top X_3^{-1}=-X_1^{-1}X_2^\top X_3^{-1}.
\]
Hence $\beta(c_3\cdot e,0,\infty)=\sgn X_1^{-1}X_2^\top X_3^{-1} = \sgn X_3(X_2^\top)^{-1}X_1$. This finishes the proof.
\end{proof}

\subsection{Fixed Points of Generators of Maximal Representations}\label{SecNormalForm}
From Proposition \ref{Prop03} we know that image of a standard generator, $\varrho(C_i)$, under  a maximal representation $\varrho:\Gamma_{0,3}\rightarrow \Sp(2n,\R)$ is conjugate to  
\begin{equation}\label{EqStandard}
   c=
     \left(\begin{array}{cc}
	A&0\\
	A+(A^\top)^{-1}S &(A^\top)^{-1}
\end{array}\right)\in \Sp(2n,\R),
\end{equation}
with $A$ invertible and $S$ symmetric definite. Throughout this section we will assume that $c$ has this form. It has at least one fixed point in $\check S$, but maybe more. We will show that $c$ has a unique fixed point in $\check S$, in which it acts non-expandingly, i.e $dc|_Y$ has no eigenvalue of absolute value bigger than $1$. The same proof can be used to show that $c$ has a unique fixed point in which it acts non-contractingly. The fixed non-attracting fixed point and the non-repellent fixed point are transversal if and only if $A$ has no eigenvalue of absolute value $1$. We use the unique non-expanding fixed point as the \emph{canonical fixed point}.

\begin{remark}\label{RemFPV}
  Recall that all fixed points of $c$ are in $V$. Indeed, according to Proposition \ref{Prop03} $c$ can appear as the image of a standard generator under a maximal representation $\varrho$ of $\Gamma_{0,3}$, say $c=\varrho(C_1)$. We can assume that a fixed point of, say, $\varrho(C_3)$ is equal to $\infty$. By Formula \eqref{Formula} every fixed point of $c$ is transverse to $\infty$, hence is contained in $V$.
\end{remark}

\begin{remark}
  In the sequel we sometimes write $C$ for $A+(A^\top)^{-1}S$.
\end{remark}

 In Section \ref{SecHyp} and \ref{SecPar} we introduce and discuss $S$-hyperbolic respectively $S$-parabolic isometries and we prove Proposition \ref{PropNormalForm} for these two cases before we conclude the general case in Section \ref{SecProofNormalForm}.
 
Now we recall briefly some  facts used later, fix some terminology and show how to construct fixed points for $c$.
 
The equation for a fixed point $Y\in \Sym_n(\R)$ of $c$ is
\begin{equation}\label{EqFixedPoint}
  Y(A+(A^\top)^{-1}S)Y+Y(A^\top)^{-1}-AY=0.
\end{equation}

\begin{remark}\label{RemOEDiag}
  Later we will sometimes assume that certain matrices $Y\in \Sym(n,\R)$ have the special form $\left(
\begin{array}{cc}
	Y_1&0\\
	0&0
\end{array}\right)$ with $Y_1\in \Sym(k,\R)$ diagonal and invertible. This is allowed since every element of $V$ is  a symmetric matrix. Hence there exists  $k\in \On(n)$ such that $kYk^{-1}$ has this form. Furthermore $l:=\left(
\begin{array}{cc}
	k&0\\
	0&k
\end{array}\right)\in \Sp(2n,\R)$ and if $Y$ is a fixed point for $g=\left(
\begin{array}{cc}
	A&0\\
	C&D
\end{array}\right)$, then $kYk^{-1}$ is a fixed point for 
\[
  lgl^{-1}=\left(
\begin{array}{cc}
	kAk^{-1}&0\\
	kCk&kDk^{-1}
\end{array}\right).
\] 
Clearly the spectrum of $A$ is equal to the spectrum of $kAk^{-1}$.
\end{remark}
We can use a block structure of $A$ to construct fixed points of $c=\left(
\begin{array}{cc}
	A&0\\
	C&D
\end{array}\right)$ in $\check S$. 
\begin{lemma}\label{LemBlockForm}
\begin{enumerate}
\item  Let $A= \left(
\begin{array}{cc}
	A_1&A_2\\
	0&A_4
\end{array}\right)\in \GL(n,\R)$, such that $A_1$ is a $k\times k$-matrix. Write $S= \left(
\begin{array}{cc}
	S_1&S_2\\
	S_3&S_4
\end{array}\right)$, where $S_1$ has the same size as $A_1$. Define
\[
  d= \left(
\begin{array}{cc}
	A_1&0\\
	A_1+(A_1^\top)^{-1}S_1 &(A_1^\top)^{-1}
\end{array}\right)\in \Sp(2k,\R).
\]
Let $Y_1\in \Sym_k(\R)$ be a fixed point for $d$. Then
 $Y=\left(
\begin{array}{cc}
	Y_1&0\\
	0&0
\end{array}\right)$ is a fixed point for $c$.
\item Conversely if  $ \left(
\begin{array}{cc}
	Y_1&0\\
	0&0
\end{array}\right)$ is a fixed point for $d$ with $Y_1$ invertible, then $A=\left(
\begin{array}{cc}
	A_1&A_2\\
	0&A_4
\end{array}\right)$, where $A_1$ has the same size as $Y_1$ and $Y_1$ is a fixed point for $c_1$ defined as in (i).
\item Let $Y$ be a fixed point of $c$. Then the differential of $c$ in $Y$ is
\[
dc|_Y(v):v\mapsto (-YC+A)v(CY+D)^{-1}.
\]
\end{enumerate}
\end{lemma}

\begin{corollary}\label{CorDifferentialFP} Using the notation from Theorem \ref{MainThmIntro} we have:
  \[
  dc_1|_0(v)=X_1vX_1^\top,\quad dc_2|_e(v)=X_2vX_2^\top,\quad dc_3|_\infty(v)=X_3vX_3^\top.
\]
\end{corollary}

\begin{remark}
  The matrix $d$ can appear as an image of a standard generators of a maximal representations into $\Sp(2k,\R)$ (see Proposition \ref{Prop03}).
\end{remark}

\begin{proof}[Proof of Lemma \ref{LemBlockForm}]
\begin{enumerate} 
\item 
  For statement (i) note that $S_1$ is positive definite symmetric because $S$ is. The verification that $Y$ is a fixed point point of $c$ is straight forward. Indeed, inserting  $\left(
\begin{array}{cc}
	Y_1&0\\
	0&0
\end{array}\right)$ in the fixed point equation \eqref{EqFixedPoint} for $c$ gives
\begin{align*}
        &   \left(
\begin{array}{cc}
	Y_1 A_1 Y_1&0\\
	0&0
\end{array}\right)
+    
 \left(
\begin{array}{cc}
	Y_1 ((A^\top)^{-1}S)_1Y_1&0\\
	0&0
\end{array}\right)\\
&+       \left(
\begin{array}{cc}
	Y_1 (A^\top)^{-1}_1&Y_1 ((A^\top)^{-1})_2 \\
	0&0
\end{array}\right)
-
       \left(
\begin{array}{cc}
	A_1 Y_1&0 \\
	0&0
\end{array}\right).
  \end{align*}
This is equal to $0$, since $Y_1$ is a fixed point for $c_1$ and the lower left block of $A$ is $0$.
  \item follows from the same equation. The matrix $A_3$ is equal to $0$ since $Y_1$ is invertible.
  \item  We calculate the  differential of $c$ at some point $Y\in V$. We know:
\[
  c:Y\mapsto AY(CY+D)^{-1}.
\]
First we calculate a power series for the map $v \mapsto (C(Y+v)+D)^{-1}$ for small $v\in V$. We abbreviate $M:=CY+D$.
\begin{align*}
  (C(Y+v)+D)^{-1}=& (M+Cv)^{-1}=(1+M^{-1}Cv)^{-1}M^{-1}\\
=&\sum_{i=0}^\infty (-M^{-1}Cv)^iM^{-1},
\end{align*}
where we where allowed to use the geometric series for matrices since we asked $v$ to be small.

Therefore we get
\begin{align*}
 & A(Y+v)(C(Y+v)+D)^{-1}\\=&AY(C(Y+v)+D)^{-1}+Av(C(Y+v)+D)^{-1}\\
=& \sum_{i=0}^\infty AY(-M^{-1}Cv)^iM^{-1}+\sum_{i=0}^\infty Av(-M^{-1}Cv)^iM^{-1}
\end{align*}
and the differential in the point $Y$ is:
\[
  dc|_Y(v)=(-AY(CY+D)^{-1}C+A)v(CY+D)^{-1}.
\]
If $Y$ is a fixed point of $c$ this is:
\[
   dc|_Y(v)=(-YC+A)v(CY+D)^{-1}.
\]
\end{enumerate}
\end{proof}

\subsubsection{Proof of Proposition \ref{PropNormalForm} for $S$-hyperbolic Isometries}\label{SecHyp}

We define
\begin{definition}\label{DefHypGen2}
 Let $G$ be a Hermitian Lie group and $g\in G$. Then $g$ is \emph{Shilov-hyperbolic} (or $S$-hyperbolic) if it has a pair $(g^+,g^-)$ of transversal fixed points in $\check S$, such that $g$ contracts an open and dense subset of $\check S$ to $g^+$ and $g^{-1}$ contracts an open and dense subset to $g^-$. Note that the fixed points $c^+$ and $c^-$ are uniquely determined.
\end{definition}

\begin{proposition}\label{PropContrFixedPoints}
Let $c$ be as in \eqref{EqStandard}. Assume $\sigma(A)\cap S^1 =\emptyset$. Then $c$ is S-hyperbolic.
\end{proposition}

Before we give the general proof, we prove two special cases:
\begin{lemma}\label{LemContrExp}  Let $c$ be as in \eqref{EqStandard}. 
\begin{enumerate}
  \item Assume that $A$ only has eigenvalues of absolute value strictly smaller than $1$. Then $c$ has a unique fixed $Y$ point transversal to $0$. It satisfies $\beta(Y,0,\infty)=n$. The action of $c$ in $Y$ is expanding and the differential $dc|_{Y}$ acts as on $T_{Y}V$ as
  \[
     dc|_{Y}:v\mapsto(Y(A^\top)^{-1}Y^{-1})c(Y(A^\top)^{-1}Y^{-1})^\top.
  \] 
  \item If $A$ only has eigenvalues of absolute value strictly bigger that $1$, then $c$ has a unique fixed point $Y$ transversal to $0$. It satisfies $\beta(0,Y,\infty)=n$. The action of $c$ in $Y$ is contracting and the differential $dc|_{Y}$ acts as on $T_{Y}V$ as
  \[
     dc|_{Y}:v\mapsto (Y(A^\top)^{-1}Y^{-1})v(Y(A^\top)^{-1}Y^{-1})^\top.
  \]
  \end{enumerate} 
  Furthermore in both cases $Y$ is invertible and depends continuously on $c$. 
\end{lemma}

\begin{proof} \begin{enumerate}
\item We are searching for a fixed point $Y$ transversal to $0$, hence we search for an invertible one. We can reformulate the fixed point equation \eqref{EqFixedPoint} to $A^\top Y^{-1}A-Y^{-1}=\bar S$, where $\bar S:=A^\top A+S$ is positive definite symmetric.

 One verifies easily that 
  \[
    Y^{-1}=-\sum_{i=0}^\infty (A^\top)^i \bar S A^i,
  \]
  is a solution, which is clearly negative definite. The sum converges since $A$ is contracting. Furthermore it is unique because the equation for $Y^{-1}$ is a linear matrix equation \cite[Ch.4.3]{Matrix} which has a unique solution if and only if for any eigenvalues $\lambda$ and $\mu$ of $A$, $\lambda\mu\neq 1$. Here this is clearly true by assumption. Furthermore $Y$ depends continuously on $c$.
  
From Lemma \ref{LemBlockForm} we know that if $Y$ is a fixed point of $c$, then:
\[
   dc|_Y(v)=(-YC+A)v(CY+D)^{-1}.
\]
For $Y=0$ we have $dc|_0(v)=AvA^\top$ and for general $Y$ we get, using the fixed point formula,
\begin{equation}\label{EqFP2}
  A-YC=Y(A^\top)^{-1}Y^{-1}\text{ and }(CY+D)^{-1}=Y^{-1}A^{-1}Y,
\end{equation}
hence 
\[
   dc|_Y(v)=(Y(A^\top)^{-1}Y^{-1})v(Y(A^\top)^{-1}Y^{-1})^\top.
\]
Therefore $c$ is expanding in $Y$.
  
  \item Analogously.
  \end{enumerate}
\end{proof}

\begin{proof}[Proof of Proposition \ref{PropContrFixedPoints}]
  We can assume that 
  \[
     A=
\left(
\begin{array}{cc}
	A_1&0\\
0&A_4
\end{array}\right),
  \]
such that the eigenvalues of the $k\times k$-matrix $A_1$ have absolute value strictly bigger than $1$ and the eigenvalues of $A_4$ have absolute value strictly less than $1$. We use Lemma \ref{LemContrExp} and Lemma \ref{LemBlockForm}  to construct the desired fixed points $X^+$ and $X^-$. By Lemma \ref{LemContrExp}
  \[
    \left(
\begin{array}{cc}
	A_1&0\\
A_1+(A_1^\top)^{-1}S_1&(A_1^\top)^{-1}
\end{array}\right) \in \Sp(2k,\R)
\]
and
\[
 \left(
\begin{array}{cc}
	A_4&0\\
A_4+(A_4^\top)^{-1}S_4&(A_4^\top)^{-1}
\end{array}\right)\in \Sp(2(n-k),\R)
  \]
  have fixed points $Y_1$ resp. $Y_4$ transversal to $0$ in their respective Shilov boundaries. By Lemma \ref{LemBlockForm} 
  \[
    y_1:=\left(
\begin{array}{cc}
	Y_1&0\\
  0&0
\end{array}\right),\quad y_4:=\left(
\begin{array}{cc}
	0&0\\
0&Y_4
\end{array}\right)
  \]
  are fixed points for $c$, and since $Y_1\in \Sym_k(\R)$ (with $k$ as above) and $Y_4 \in \Sym_{n-k}(\R)$ are invertible, $y_1-y_4$ is invertible, hence $y_1$ and $y_4$ are transversal. 
  
By Lemma \ref{LemBlockForm} (iii) we know that 
 \[
   d{c}|_{y_1}:v\mapsto (-y_1C+A)v(Cy_1+D)^{-1}.
 \] 
A straight forward calculation shows that 
\[
  -y_1C+A=\left(
\begin{array}{cc}
	A_1-Y_1C_1 &Y_1C_2\\
	0&A_4.
\end{array}\right)
\]
Since $Y_1$ is invertible and a fixed point for $c_1$ we get from the fixed point equation \eqref{EqFixedPoint}:
\[
A_1-Y_1C_1=Y_1^{-1}(A_1^\top)^{-1}Y_1,
\]
hence 
\[
  -y_1C+A=\left(
\begin{array}{cc}
	 Y_1^{-1}(A_1^\top)^{-1}Y_1&Y_1C_2\\
	0&A_4.
\end{array}\right)
\]
The eigenvalues are the eigenvalues of $(A_1^\top)^{-1}$ and $A_4$, hence it is contracting

The same calculation shows
\[
  (Cy_1+D)^{-1}=(-y_1C+A)^\top.
\]
  Hence $c$ acts contracting in $y_1$ Along the same lines one can show that $y_4$ is a repellent fixed point for $c$.
  
  Let $g$ be an isometry which maps $(0,\infty)$ to $(y_4,y_1)$. Then $gcg^{-1}=\left(
\begin{array}{cc}
	\bar A&0\\
	0&(\bar A^\top)^{-1}
\end{array}\right)$ where $\bar A$ is a conjugate to $-y_1C+A$, hence contracting. Therefore $gcg^{-1}$ contracts $V$ to $0$ and since $V$ can be seen as an open and dense subset of the Shilov boundary $\check S$ of the bounded symmetric space associated with $\Sp(2n,\R)$, this finishes the proof (c.f. Remark \ref{RemVopendense}).
\end{proof}

\subsubsection{Proof of Proposition \ref{PropNormalForm} for $S$-parabolic Isometries}\label{SecPar}
\begin{definition}
 Let $G$ be a Hermitian Lie group. Then $g\in G$ is \emph{Shilov-parabolic} or $S$-parabolic if $g$ has a unique fixed point in $\check S$. 
\end{definition}

\begin{proposition}\label{PropPar}
 Let $c$ be as in \eqref{EqStandard}. Assume $\sigma(A)\subset S^1$. Then  $c$ is S-parabolic.
 \end{proposition}
\begin{proof} First note that by Remark \ref{RemFPV} all fixed points of $c$ are in $V$. So let $Y\in V$ be a fixed point. 
  Since $Y$ is a symmetric matrix we can assume without loss of generality that 
  \[
    Y=  \left(
\begin{array}{cc}
	Y_1&0\\
	0&0
\end{array}\right)
  \]
  such that $Y_1$ is a square matrix and diagonal invertible. 
Now we can apply Lemma \ref{LemBlockForm} (ii). Therefore $A$ decomposes into a block form and the eigenvalues of $A_1$ also have absolute value $1$.  Multiplying both sides of the fixed point equation \eqref{EqFixedPoint} for $Y_1$ from the left with $A_1^\top Y_1^{-1}$ and from the right with $Y_1^{-1}$ (which we are allowed to, since $Y_1$ was chosen to be invertible), we get

 \begin{equation}\label{EqPosDef}
  A_1^\top Y_1^{-1} A_1 -Y_1^{-1}=A_1^\top A_1 +S_1.  
 \end{equation}
Choose an eigenvector $v\neq 0$ for a (possibly complex) eigenvalue $\lambda$ of $A_1$ with $|\lambda |=1$. Since $A_1$ is a real matrix, we have $A_1^\top=A^*$ and hence
  \begin{align*}
   v^*A_1^\top Y_1^{-1} A_1v -v^*Y_1^{-1}v=&v^*A_1^* Y_1^{-1} A_1v -v^*Y_1^{-1}v\\ 
   =&\underbrace{\bar \lambda  \lambda}_{=1} v^*Y_1^{-1}v  -v^*Y_1^{-1}v=0,
 \end{align*}
 which shows that the left hand side of \eqref{EqPosDef} is indefinite. But the right hand side is positive definite so we get a contradiction. Hence $0$ is the only fixed point of $c$. 
\end{proof}

\subsubsection{Proof of Proposition \ref{PropNormalForm}}\label{SecProofNormalForm}

We have proven the Proposition \ref{PropNormalForm} for  $S$-hyperbolic $c$ in Section \ref{SecHyp} and for $S$-parabolic $c$ in Section \ref{SecPar} and use this to write down the desired fixed point explicitly. Note that we use here that there are $\Sp(2n,\R)$ contains copies of $\Sp(2k,\R)$ for $k\leq n$.  This statement does not hold in an analogous form for other Hermitian Lie groups (e.g. the exceptional one), whence this proof can not be generalized one-to-one.
\begin{proof}[Proof of Proposition \ref{PropNormalForm}]
(i) follows immediately from Proposition \ref{PropPar}.

For (ii) we have to combine methods from the last two subsections. 
  As in the proof of Proposition  \ref{PropContrFixedPoints} we can assume that $A$ is of block form $\left(
\begin{array}{cc}
	A_1&0\\
	0&A_4
\end{array}\right)$ where the eigenvalues of $A_2$ all have absolute value $1$ and the absolute values of the eigenvalues of $A_1$ are different from $1$. If $A_1$ is a $k\times k$ matrix, we denote by $C_1$ the upper left $k \times k$ block in the $n\times n$-matrix $A+(A^\top)^{-1}S$. Then $C_1=A_1+(A_1^\top)^{-1} S_1$, where $S_1$ is the upper left $k\times k$ block of the right size of $S$; it is automatically symmetric positive definite. Then 
\[
 c_1= \left(
\begin{array}{cc}
	A_1&0 \\
	C_1&(A_1^\top)^{-1}
\end{array}\right)\in \Sp(2k,\R)
\]
is by construction hyperbolic. Hence it has a unique fixed point $Y_1$ in $V_k=\Sym(k,\R)$ in which the action of $c_1$ is contracting. By Lemma \ref{LemBlockForm}
$Y=\left(\begin{array}{cc}
	Y_1&0 \\
	0&0
\end{array}\right)\in V$ is a fixed point of $c$ and $c$ acts non-expandingly in $X$.

Now it remains to show that this is the unique fixed point with this property. After conjugating $c$ with $g=\left(
\begin{array}{cc}
	1 &-Y \\
	0&1
\end{array}\right),
$
we can assume that $0$ is a non-repellent fixed point. Let $\bar Y$ be another fixed point. Again, after eventual conjugation with an isometry $h\in \On(n)$ (which stabilizes $0$) we can assume that $\bar Y=\left(\begin{array}{cc}
	\bar Y_1&0 \\
	0&0
\end{array}\right) $, with $\bar Y_1$ invertible. By Lemma \ref{LemBlockForm} (ii)
\[
  hgc_1(hg)^{-1}=\left(
\begin{array}{cc}
	\bar A &0 \\
	\bar C& (\bar A^\top)^{-1}
\end{array}\right),
\] 
with
\[
  \bar A=\left(
\begin{array}{cc}
	\bar A_1 &\bar A_2 \\
	0& \bar A_4
\end{array}\right)\text{ and }\bar C=\left(
\begin{array}{cc}
	\bar C_1 &\bar C_2 \\
	\bar C_3& \bar C_4
\end{array}\right)
\] 
and $\bar Y_1$ is a fixed point of $\bar c_1=\left(
\begin{array}{cc}
	\bar A_1 & 0 \\
	\bar C_1& (\bar A_1^\top)^{-1}
\end{array}\right)$. Since $\bar Y_1$ is invertible we can calculate $d\bar c_1|_{\bar Y_1}$ as in the proof of Lemma \ref{LemContrExp} and we get
\[
  d\bar c_1|_{\bar Y_1}:v\mapsto(\bar Y_1(A_1^\top)^{-1}\bar Y_1^{-1})c(\bar Y_1(A^\top)^{-1}\bar Y_1^{-1})^\top
\]
and the action of $\bar c_1$ in $\bar Y_1$ is expanding. Hence $c$ has at least one expanding direction in any fixed point different from $0$ and $0$ is the only non-repellent fixed point. This finishes the proof.
\end{proof}

\section{Gluing}\label{SecGlue}
\subsection{Gluing in $\Sp(2n,\R)$}\label{SecGlueSp2nR}
In this section we give more details on the gluing construction and prove Theorem \ref{PropGlueSp2nR} below, which is an extension of Theorem \ref{PropGlueSp2nRIntro}.

There are two gluing constructions: gluing two surfaces and closing handles. 
\begin{figure}[ht]
\includegraphics[width=\textwidth]{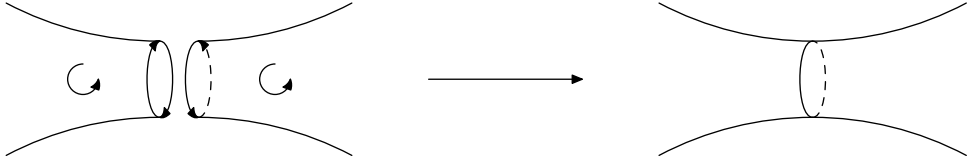}
\caption{Gluing two surfaces.}
\label{GluingIntro}
\end{figure}

\begin{figure}[ht]
\centering
\includegraphics{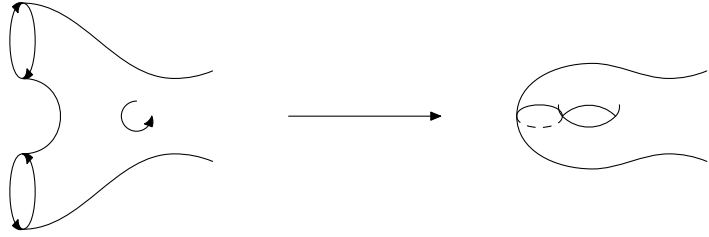}
\caption{Closing a handle.}
\label{ClosingIntro}
\end{figure}

To obtain an oriented surface whose orientation is compatible with the building blocks the gluing homeomorphism has to be orientation reverting (see Figures \ref{GluingIntro} and \ref{ClosingIntro}) along the boundary curves, which we denote by $C$ and $\bar C$.

Let $\Sigma$ and $\bar \Sigma$ be surfaces and $C\subset \Sigma$ resp. $\bar C\subset \bar \Sigma$ boundary components. Let $\varrho:\pi_1(\Sigma)\rightarrow G$ and $\bar \varrho:\pi_1(\bar \Sigma)\rightarrow G$ representations into some group $G$. Let $\Sigma'$ be the surface obtained by gluing $\Sigma$ and $\bar \Sigma$ along $C$ and $\bar C$. Recall that a representation $\varrho':\pi_1(\Sigma')\rightarrow G$ is said to be \emph{glued} from $\varrho$ and $\bar \varrho$ if $\varrho'|_{\pi_1(\Sigma)}=\varrho$ and $\varrho'|_{\pi_1(\bar \Sigma)}=\bar \varrho$. Since we glue orientation preservingly we identify $C$ with the inverse of $\bar C$ (seen as elements of $\pi_1(\Sigma')$). Therefore $\varrho$ and $\bar \varrho$ can be glued along $C$ and $\bar C$ if and only $\varrho(C)=\bar\varrho(\bar C)^{-1}$. The fundamental group of $\Sigma'$ is an amalgam of the ones of $\Sigma$ and $\bar \Sigma$.

Conjugacy classes $[\varrho]$ and $[\bar \varrho ]$ of representations can be glued if and only $\varrho(C)$ and $\bar\varrho(\bar C)^{-1}$ are conjugate.

Analogously one can show that closing a handle for a surface $\bar \Sigma$ along $C_1$ and $C_2$ with a resulting surface $\Sigma$ corresponds to an HNN extension:
\[
  \pi_1(\Sigma)=\langle \pi_1(\bar \Sigma),t| C_2^{-1}=tC_1t^{-1} \rangle.
\] 
Again this shows that we can close a handle if and only if $\varrho(C_1)^{-1}=g\varrho(C_2)g^{-1}$.

\begin{theorem}\label{PropGlueSp2nR}
 Let $\varrho:\pi_1(\Sigma) \rightarrow \Sp(2n,\R)$ and $\bar \varrho:\pi_1(\bar \Sigma) \rightarrow \Sp(2n,\R)$ be (non-necessarily distinct) maximal representation with distinct boundary components $C\subset \Sigma$ and $\bar C\subset \bar \Sigma$. 
\begin{enumerate}
\item
We can conjugate $\varrho$ and $\bar \varrho$ such that
  \begin{equation}\label{EqcNormalform2Intro}
  c:=\varrho(C)=\left(\begin{array}{cc}
	X&0\\
	X+(X^\top)^{-1}S &(X^\top)^{-1}
\end{array}\right)
\end{equation} 
and 
\begin{equation}
\bar c:=\bar \varrho(\bar C)=\left(
\begin{array}{cc}
	(\bar X^\top)^{-1}&-(\bar X^\top)^{-1}-\bar S\bar X \\
	0& \bar X
\end{array}\right)\end{equation}
with $X$ and $\bar X$ invertible and $S$ and $\bar S$ symmetric positive definite. 
\item 
The representations $\varrho$ and $\bar \varrho$ can be glued along $C$ and $\bar C$ if and only if $\varrho(C)$ and $\bar \varrho(\bar C)^{-1}$ are conjugate in $\Sp(2n,\R)$.
 \item  Suppose $X$ and $\bar X$ contracting. 
Then $\bar c$ and $c^{-1}$ are conjugate in $\Sp(2n,\R)$ if and only $X^\top$ and $\bar X$ are conjugate in $\GL(n,\R)$. If $ \bar X=G X^\top G^{-1}$, then $\bar c=gc^{-1}g^{-1}$ with
  \[
  g=\left(\begin{array}{cc}
	\bar Y GY^{-1}-(G^\top)^{-1}& -\bar YG\\
	GY^{-1} & -G
\end{array}\right),
\]  
where \[ 
Y=-\left(\sum_{i=0}^\infty (X^\top)^i (X^\top\cdot X+S)X_1^i \right)^{-1}
\]
and 
\[
 \bar Y=\sum_{i=0}^\infty (\bar X^\top)^i(I+\bar X^\top \bar S \bar X)\bar X^i.
\]  

\item It $X$ or $\bar X$ has an eigenvalue of absolute value $1$, then $\bar c$ and $c^{-1}$ are not conjugate in $\Sp(2n,\R)$.
\end{enumerate}
\end{theorem}

\begin{proof}[Proof of Theorem \ref{PropGlueSp2nR}]
  \begin{enumerate}
  \item Follows from Theorem \ref{MainThmIntro}. 
  
  \item Follows from the discussion in the beginning of Section \ref{SecGlueSp2nR}.
  
  \item First note that since $c$ and $\bar c$ are $S$-hyperbolic, they have fixed points $Y$ resp. $\bar Y$ which are transversal to $0$ resp. $\infty$ (Lemma \ref{LemContrExp}. We are searching for $g$ with $\bar c=gc^{-1}g^{-1}$. We want to write $g=g_1g_2g_3$, where the $g_i$ have the following properties:  $g_3$ maps the transverse pair $(0,Y)$ to $(0,\infty)$, $g_2$ fixes $(0,\infty)$ and $g_1$ maps the transverse pair $(0,\infty)$ to $(\infty,\bar Y)$. We choose
  \[
    g_1:=\left(\begin{array}{cc}
	\bar Y&-1\\
	1 &0
\end{array}\right)
  \]
  and
    \[
    g_3:=\left(\begin{array}{cc}
	Y^{-1}&-1\\
	1 &0
\end{array}\right).
  \]
  Then 
  \[
    g_1^{-1}\bar cg_1= \left(\begin{array}{cc}
	\bar X &\\
	 &(\bar X^\top)^{-1}
\end{array}\right)
  \]
  and
   \[
    g_3 c ^{-1}g_3^{-1}= \left(\begin{array}{cc}
	X^\top&\\
	 &X^{-1}
\end{array}\right)  \] 
  By assumption $\bar X$ as well as $X^\top$ are contracting. Hence, if there exists $g_2$ such that $g_1^{-1}\bar cg_1=g_2 g_3 c^{-1}g_3^{-1} g_2^{-1}$, then $g_2$ has to stabilize the pair $(0,\infty)$, i.e.
  \[
    g_2=   \left(\begin{array}{cc}
	G&\\
	 &(G^\top)^{-1}
\end{array}\right) 
  \]
   In particular there has to be a $G\in \GL(n,\R)$ such that $\bar X=G X^\top G^{-1}$. Then $g=g_1g_2g_3$.

    \item  From Theorem \ref{PropGlueSp2nRIntro} (i) we know that we can assume
      \[
   c=\left(
\begin{array}{cc}
	X&0 \\
	M&(X^\top)^{-1}
\end{array}\right),\quad \bar c=\left(
\begin{array}{cc}
	\bar X &0 \\
	\bar M&(\bar X^\top)^{-1}
\end{array}\right), \quad 
\]
       As explained in the second part of the proof of Theorem \ref{MainThmIntro} we can assume that $X^{-1}$ and $\bar X$ are non-expanding. Then $0$ is the unique fixed point for $c^{-1}$ and $\bar c$ where the differential is non-expanding. Hence if there exists $g\in \Sp(2n,\R)$ with $g\bar cg^{-1}=c^{-1}$, then $g$  has to fix $0$. Assume
       \[
g=\left(
\begin{array}{cc}
	A&0 \\
	X&(A^\top)^{-1}
\end{array}\right).
  \]
Assume that $g\bar cg^{-1}=c^{-1}$. Then $A \bar XA^{-1}=X^{-1}$. This is a first condition for $c^{-1}$ and $\bar c$ to be conjugate. If $\bar X$ and $X^{-1}$ are not conjugate, we are done. 
  
Now assume that $A\bar XA^{-1}=X^{-1}$. Then we can write 
\[
  g=\left(\begin{array}{cc}
	A&0 \\
	C&(A^\top)^{-1}
\end{array}\right)=\left(\begin{array}{cc}
	1&0 \\
	CA^{-1}&1
\end{array}\right)\left(\begin{array}{cc}
	A&0 \\
	0&(A^\top)^{-1}
\end{array}\right).
\]
Define $\bar C:=CA^{-1}$ and $M':= (A^\top)^{-1}\bar MA^{-1}$. Recall $A\bar XA^{-1}=X^{-1}$ and $(X^\top)^{-1} M'$ is symmetric and positive definite. 

We can summarize that to the equation
\begin{align*}
  g\bar cg^{-1}=&\left(\begin{array}{cc}
	1&0 \\
	\bar C&1
\end{array}\right)
\left(
\begin{array}{cc}
	X^{-1}&0 \\
	 M' &X^\top
\end{array}\right)
\left(\begin{array}{cc}
	1&0 \\
	-M'&1
\end{array}\right)\\
=&\left(
\begin{array}{cc}
	X^{-1}&0 \\
	\bar CX^{-1}+M' -X^\top \bar C &\bar X^\top
\end{array}\right)
\overset{!}{=}c^{-1}=\left(
\begin{array}{cc}
	X^{-1}&0 \\
	-M^\top &X^\top
\end{array}\right).
\end{align*}
In particular: $\bar CX^{-1}+ M' -X^\top \bar C=-M^\top$, which is equivalent to 
\[
  (X^\top)^{-1}\bar CX^{-1}-\bar C+(X^\top)^{-1}M'=-(X^\top)^{-1}M^\top.
\]
Note that by construction $(X^\top)^{-1}\bar M$ is positive definite, hence $-(X^\top)^{-1}M^\top$ is negative definite. Let $\lambda$ be an eigenvalue  of $X^{-1}$  with $|\lambda |=1$ and let $v$ be a non-zero eigenvector for $\lambda$. Such an eigenvalue exists by assumption. Then 
\begin{align*}
    v^*\big((X^*)^{-1}C'X^{-1}-C'+(X^*)^{-1} M'\big)v=&v^* (X^*)^{-1} M' v\\
    =&-v^* (X^*)^{-1}M^* v,
\end{align*}
which is a contradiction since the left hand side is strictly positive and the right hand side is strictly negative. Therefore $c^{-1}$ and $\bar c$ cannot be conjugate.
  \end{enumerate}
\end{proof}

\subsection{The Gluing Graph}\label{SecGraph}
To be able to state coordinates for more general surfaces with need to encode the gluing involving several pairs of pants and handles in a clear way.

Let $\Sigma_{g,m}$ be the topological surface with genus $g$ and $m\geq 1$ boundary components and $\chi(\Sigma_{g,m})<0$. It can be build using $2g-2+m$ pairs of pants.	

This gluing can be visualized in a \emph{gluing graph}. Given $\Sigma_{g,m}$ with a decomposition into pairs of pants. We construct the gluing graph for this decomposition as follows: we represent any pair of pants and any boundary component by a vertex. We add an edge between two pairs of pants with a common boundary component for each common boundary component. Furthermore we join every pair of pants with the vertices associated with its boundary components. Note that these graphs are connected.

 Here are some examples:
\begin{center}
\begin{tabular}[ht]{C{4cm}C{4cm}}
  \includegraphics[width=2cm]{pairofpants.eps} & \includegraphics[width=2cm]{graphen.1}\\
    \includegraphics[width=3cm]{s11.3}& \includegraphics[width=3cm]{graphen.2}\\
  \includegraphics[width=4.5cm]{s20.3}& \includegraphics[width=3cm]{graphen.4}
\end{tabular}
\end{center}
 Clearly the graph depends on the decomposition into pairs of pants.
\begin{center}
\begin{tabular}{C{4cm}C{4cm}}
  \includegraphics[width=4.5cm]{s20.2}&\includegraphics[width=3cm]{graphen.3}\\
\end{tabular}
\end{center}

To parametrize maximal representations of $\Gamma_{g,m}$  we will label the gluing graph with the length and twist parameters. 

  \begin{figure}[ht]
  \centering
  \includegraphics{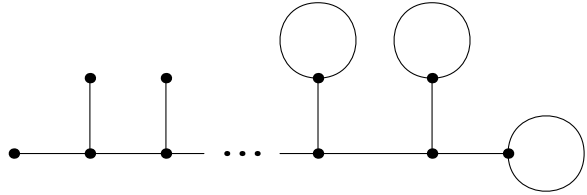}
\caption{Standard graph}
\label{GrStandardGraph}
\end{figure}

\begin{definition}\label{DefStandardGraph}
  Let $\Sigma_{g,m}$ be a surface with negative Euler characteristic. Then we call the  decomposition into pairs of pant as in Figure \ref{GrStandardGraph}
   \emph{standard decomposition}. This graph is the \emph{standard graph}.
\end{definition}

\section{Coordinates for Maximal Representations}\label{SecMoreParam}
We use the gluing graph introduced in the previous section to state general parameters for $\Rep_{max}(\Gamma_{g,m},\Sp(2n,\R))$. In Section \ref{Sec2Pairsofpants} we parametrize maximal representations of $\Gamma_{0,3}$, $\Gamma_{1,1}$, $\Gamma_{0,4}$, $\Gamma_{1,2}$ and $\Gamma_{2,0}$. In Section \ref{SecParamGen} we define the coordinates for the most general case $\Gamma_{g,m}$.

\begin{remark}\label{RemTwist}
   The length parameters from Theorem \ref{MainThmIntro} are only unique up to conjugation with an element of $\On(n)$. For the gluing of two representations we have to choose representatives from these equivalence classes and glue them. The conjugation class of the resulting representation does not depend on this choice. Indeed, replace $X$ and $S$ by $kXk^{-1}$ and $kSk^{-1}$ resp. $\bar X$ and $\bar S$ by $\bar k \bar X \bar k^{-1}$ and $\bar k \bar S \bar k^{-1}$ and the $G$ is replaced by $\bar k G k^{-1}$, and the resulting representation from both sets of parameters are conjugate.  
\end{remark}

Recall
\[
  B=\{X\in \GL(n,\R) | X \text{ contracting} \}.
\]
  and
  \begin{align*}
     R:= \{(X_1,X_2,X_3)\in \bar B^3 |& X_3(X_2^\top)^{-1}X_1 \text{ is symmetric}\\ 
             &\text{and positive definite}\}.
  \end{align*}

\subsection{Surfaces obtained from one or two Pairs of Pants}\label{Sec2Pairsofpants}
In Theorem \ref{MainThmIntro} we identified $\Rep_{max}(\Gamma_{g,m},\Sp(2n,\R))$ and $R/\On(n)$.

Recall $\Gamma_{1,1}=\langle A,B,C| [A,B]C\rangle$.
\begin{proposition}\label{PropG11}
 There exists a bijection between 
 \[
   \big\{ (X_1,X_2,G)\in \GL(n,\R)^3 | X_1\in B, ~ (X_1,X_2,GX_1^\top G^{-1})\in R \big\}/\On(n), 
 \]
 and $\Rep_{max}(\Gamma_{1,1},\Sp(2n,\R))$ induced by the map which assigns to $(X_1,X_2,G)$ the representation $\varrho$ defined by:
  \begin{align*}
  \varrho(A)&=\left(\begin{array}{cc}
	X_1&0\\
	X_1+(X_1^\top)^{-1}S &(X_1^\top)^{-1}
\end{array}\right),\\
  \varrho(B)&=\left(\begin{array}{cc}
	\bar Y GY^{-1}-(G^\top)^{-1}& -\bar YG\\
	GY^{-1} & -G
\end{array}\right),\\
\varrho(C)&=\left(
\begin{array}{cc}
  C_1 &C_2\\
  C_3 & C_4
\end{array}\right)
\end{align*}
with 
\begin{align*}
	C_1=&-(X_2^\top)^{-1}X_1(S^\top)^{-1} X_1^\top-X_2-(X_2^\top)^{-1}\\
   C_2=& X_2 + (X_2^\top)^{-1}X_1(S^\top)^{-1} X_1^\top\\
	C_3 =&-(X_2^\top)^{-1}X_1(S^\top)^{-1} X_1^\top - (X_2^\top)^{-1}\\
    C_4=&(X_2^\top)^{-1}X_1(S^\top)^{-1} X_1^\top\\
    S=&X_1^\top X_2^{-1}(G^\top)^{-1}X_1G^\top\\
    Y=&-\left(\sum_{i=0}^\infty (X_1^\top)^i (X_1^\top\cdot X_1+S)X_1^i \right)^{-1}\\
    \bar Y=&(G^\top)^{-1}\left(\sum_{i=0}^\infty (X_1)^i (G^\top \cdot G+X_1G^\top S^{-1}GX^\top)(X_1^\top)^i\right)G^{-1}.
\end{align*}
\end{proposition}
\begin{proof}[Proof of Proposition \ref{PropG11}]
  Let $\varrho:\Gamma_{1,1} \rightarrow \Sp(2n,\R)$ be a maximal representation. Then we can define $\varrho':\Gamma_{0,3}\rightarrow \Sp(2n,\R)$ by
  \[
    \varrho'(C_1):=\varrho(A),\quad \varrho'(C_2):=\varrho(C),\quad \varrho'(C_3):=\varrho(BA^{-1}B^{-1}).
  \]
  By Theorem \ref{PropPropToledo} $\varrho'$ is a maximal representation of $\Gamma_{0,3}$.
   
We can  assume that the $\varrho'(C_i)$ are as in Theorem \ref{MainThmIntro} for some triple $(X_1,X_2,X_3)\in R$. Then $\varrho'(C_1)$ has the form of $c$ and $\varrho'(C_3)$ has the form of $\bar c$ in Proposition \ref{PropGlueSp2nRIntro} and by construction $\varrho'(C_1)^{-1}$ and $\varrho'(C_3)$ are conjugate and by the same proposition they are both hyperbolic.  Since $\varrho'(C_1)^{-1}$ and $\varrho'(C_3)$ are conjugate, there exists $G\in \GL(n,\R)$ with  $X_3=GX_1^\top G^{-1}$ and
 \[
 \varrho(B)=\left(\begin{array}{cc}
	Y_3 GY_1^{-1}-(G^\top)^{-1}&  Y_3G\\
	GY_1^{-1} & -G
\end{array}\right),
 \]
 where $Y_1$ is the fixed point of $\varrho'(C_1)$ transversal to $0$ and $Y_3$ is the fixed point of $\varrho'(C_3)$ transversal to $\infty$.
By Remark \ref{RemTwist} this triple $(X_1,X_2,G)$ is unique up to conjugation with an element from $\On(n)$.
   
  We can construct a maximal representation of $\Gamma_{0,3}$ for any triple $X_1$, $X_2$ and $G$ with $X_1$ contracting and $(X_1,X_2,(GX_1G^{-1})^\top)\in R$ and close the handle according to Proposition \ref{PropGlueSp2nRIntro}. This provides an inverse map to the construction given above.
\end{proof}
  \begin{proposition}\label{PropG04} 
  There exists a bijection between 
  \begin{align*}
   \big\{ &(X_1,X_2,X_3,\bar X_1, \bar X_2, G)\in \GL(n,\R)^6 | (X_1,X_2,X_3)\in R,\\ 
     & (\bar X_1, \bar X_2, GX_1^\top G^{-1})\in R, X_1 \text{ contracting} \big\} / \sim
  \end{align*} and $\Rep_{max}(\Gamma_{0,4},\Sp(2n,\R))$, 
  where for $k,l\in \On(n)$, 
  \[ 
  (X_1,X_2,X_3,\bar X_1, \bar X_2, G)
  \]
   and 
   \[
    (kX_1k^{-1},kX_2k^{-1},kX_3k^{-1},l\bar X_1l^{-1}, l\bar X_2 l^{-1}, lGk^{-1})
    \] 
  are equivalent.
\end{proposition}
\begin{proposition}\label{PropG12}
  There exists a bijection between 
  \begin{align*}
   \big\{(&X_1,X_2,G,Y,\bar X,H)\in \GL(n,\R)^6|  (X_1,X_2,G\bar X^\top G^{-1})\in R, \\  & (Y,\bar X,HY^\top H^{-1})\in R, Y,\bar X \text{ contracting}  \big\}/ \sim
  \end{align*}
  and   $\Rep_{max}(\Gamma_{1,2},\Sp(2n,\R))$, where for  $k,l\in \On(n)$ 
\[ 
  (X_1,X_2,G,Y,\bar X,H)
\] and 
\[
(kX_1k^{-1},kX_2k^{-1},kGl^{-1},lYl^{-1},l\bar Xl^{-1},lHl^{-1})
\]
 are equivalent. 
\end{proposition}

  \begin{proposition}\label{PropG20}
 There exists a bijection between 
  \begin{align*}
  \big\{ & (X_1,X_2,X_3,G_3, G_2,G_1)\in \GL(n,\R)^6 |   (X_1,X_2,X_3)\in R,\\ & (G_1 X_3^\top G_1 ^{-1}, G_2 X_2^\top G_2^{-1},G_3 X_1^\top G_3^{-1})\in R, \\  & X_i \text{ contracting} \big\}/ \sim
  \end{align*} and $\Rep_{max}(\Gamma_{2,0},\Sp(2n,\R))$,
  where  for $l,k\in \On(n)$, 
  \[
  (X_1,X_2,X_3,G_3, G_2,G_1)
  \]
   and 
   \[   (kX_1k^{-1},kX_2k^{-1},kX_3k^{-1},lG_3k^{-1}, lG_2k^{-1} ,lG_1 k^{-1})
   \]
   are equivalent.
\end{proposition}

\begin{remark}\label{RemG11}
  Note that the triple $(X_1,X_2,GX_1\top G^{-1})$ is an element of $\tilde R$ if and only if 
  \[
  (X_1^\top)^{-1} GX_1^\top G^{-1} (X_2^\top)^{-1}=[(X_1^\top)^{-1},G](X_2^\top)^{-1}
  \] is symmetric positive definite. 
\end{remark}

\begin{remark}\label{RemGluingGraphs} As a preparation of the most general statement of the coordinates in Theorem \ref{ThmGeneral}, we label the gluing graphs with the parameters as follows:
\begin{figure}[ht]
\centering
\includegraphics{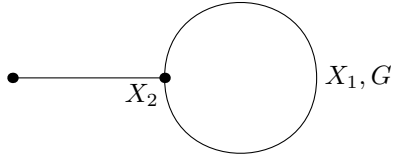}
\caption{Gluing graph for $\Gamma_{1,1}$}
\end{figure}

\begin{figure}[ht]
\centering  \includegraphics{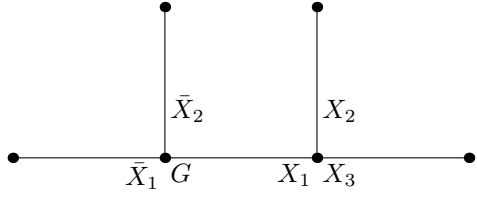}
\caption{Gluing graph for $\Gamma_{0,4}$}
\end{figure}
\begin{figure}[ht]
\centering  \includegraphics{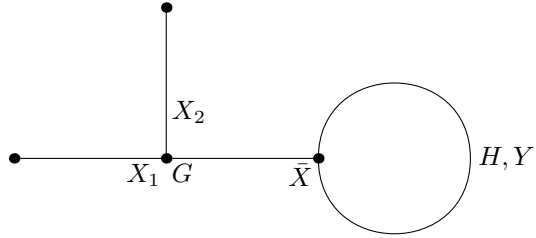}
\caption{Gluing graph for $\Gamma_{2,1}$}
\end{figure}\end{remark} 

One goal to construct explicit coordinates is to find concrete examples for representations. 

The propositions from Section \ref{SecMoreParam} can be used to draw path' in the representations varieties. Indeed, to draw a path in $\Rep_{max}(\Gamma_{0,3},\Sp(2n,\R))$, it suffices to draw a path in $R$. To do that we can choose two arbitrary path' for $X_1(t)$ and $X_2(t)$ in $\bar B$. Choose in addition a path $ S(t)\subset \Sym(n,\R)^+$ such that 
 \[
   X_3(t):=S(t)X_1(t)^{-1}X_2(t)^\top
 \] 
 is in $\bar B$. Then the $(X_1(t),X_2(t),X_3(t))$ defines a path in $R$. It is clear that the choice of a path $S(t)$ in $\Sym(n,\R)^+$ is always possible and since scaling with positive numbers does not chance the property symmetric positive definite, we can scale this path such that $S(t)X_1(t)^{-1}X_2(t)^\top\subset \bar B$.

\begin{remark}
  Unfortunately we are not able to control the eigenvalues of $X_3(t)$ in the construction presented above. While we can choose the conjugacy class of $X_1$ and $X_2$ arbitrarily, we have no control over the one of $X_3$. 
  
  We want to remark that $\GL(n,\R)$ acts on $R$ via 
  \[
     g.(X_1,X_2,X_3):= (gX_1g^\top, gX_2g^\top, gX_3g^\top).
  \]
 Let $X_3(X_2^\top)^{-1}X_1=:S$ be symmetric and positive definite. Then $S=g^{-1}(g^\top)^{-1}$ for some $g\in \GL(n,\R)$ (which can be chosen canonically!). This means that every representation of $\Gamma_{0,3}$ into $\GL(n,\R)$ gives rise to a family of maximal representation of $\Gamma_{0,3}$ into $\Sp(2n,\R)$. 
\end{remark}
 For $n=2$ we can write down more representations explicitly:
 \begin{example}
   For maximal representations $h$ from $\Gamma_{0,3}$ into $\SL(2,\R)$ we can choose the eigenvalues of all $h(C_i)$ independently. So defining 
   \[
     X_1:=h(C_1),\quad X_2:=(h(C_2)^\top)^{-1},\quad X_3:=h(C_3)
   \]
   gives parameters of a maximal representation into $\Sp(2n,\R)$ whose eigenvalues we can control. In particular, we can multiply each $X_i$ with a positive number and still obtain a maximal representation.
 \end{example}

\begin{corollary}\label{CorG11}
  There exists a surjective map from  $B\times \GL(n,\R)\times \Omega$ onto $\Rep_{max}(\Gamma_{1,1},\Sp(2n,\R))$.
\end{corollary}

\begin{corollary}\label{CorG04}
  There exists a surjective map from $B\times \GL(n,\R)^3\times \Omega^2$ onto $\Rep_\text{max}(\Gamma_{0,4},\Sp(2n,\R))$.
\end{corollary}

\begin{corollary}\label{CorG12}
  There exists a surjective map from $B^2 \times \GL(n,\R)^2\times \Omega^2$ onto $\Rep_{max}(\Gamma_{1,2},\Sp(2n,\R))$.
\end{corollary}

\subsection{General parameters}\label{SecParamGen}
 In this section we state the most general theorem for Fenchel-Nielsen coordinates for maximal representations of $\Gamma_{g,m}$ into $\Sp(2n,\R)$. 
 
 The strategy to obtain these coordinates is the same as for the examples above. 
 \begin{enumerate}
   \item Choose a decomposition of the underlying surface into pairs of pants and handles and write down the corresponding gluing graph,
   \item  Theorem \ref{MainThmIntro} gives us coordinates for representations of $\Gamma_{0,3}$ and Proposition \ref{PropG11} gives coordinates for representations of $\Gamma_{1,1}$,
   \item From Proposition \ref{PropGlueSp2nRIntro} and Remark \ref{RemTwist} we know in which cases we can  glue representations and how we get twist parameters.
 \end{enumerate}

 \begin{theorem}\label{ThmGeneral}
   Let $\varrho:\Gamma_{g,m} \rightarrow \Sp(2n,\R)$ be a maximal representation. Then there exist length and twist parameters as in the following gluing graph:
\begin{figure}[ht]
\includegraphics[width=\textwidth]{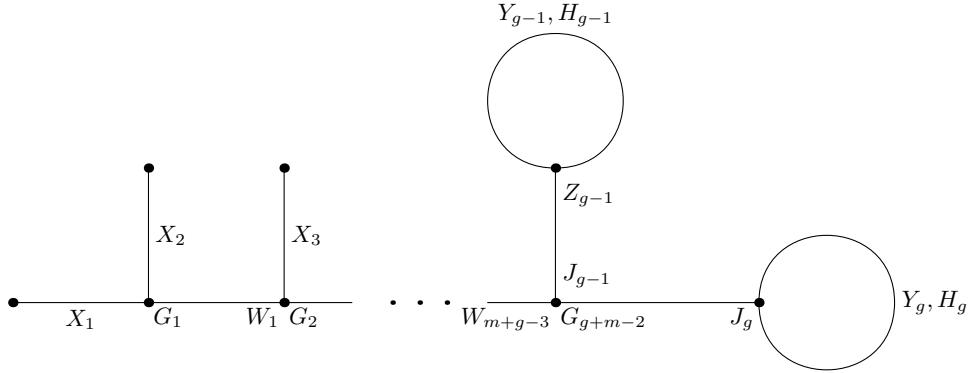}
\caption{General gluing graph}
\label{GenGluing}
\end{figure}
   
   where the $W_\bullet,X_\bullet, Y_\bullet, Z_\bullet$ are length parameter and the $G_\bullet, H_\bullet, J_\bullet$ are twist parameters subject to the usual relations and identifications.
   
   Conversely any representations defined with these parameters is maximal.
 \end{theorem}

\begin{remark}\label{RemPath}
  Theorem \ref{ThmGeneral} will be used in Section \ref{SecConComp} to define paths in $\Rep_{max}(\Gamma_{g,m}, \Sp(2n,\R))$. 
\end{remark}

\section{Applications}\label{SecApp}
\subsection{Limit Curves}
\begin{theorem}\label{ThmContLC}
 Let $h:\Gamma_{g,m}\rightarrow \PSL(2,\R)$ be a hyperbolization for a surface with geodesic boundaries. Denote by $\mathcal L$ its limit set in $S^1$. 
  Let $\varrho:\Gamma_{g,m} \rightarrow \Sp(2n,\R)$ be a maximal representation s.t. $\varrho(C_i)$ is $S$-hyperbolic for all $i$. Then $\varrho$ is Anosov and there exists a monotone, $\varrho$-equivariant, continuous map
  \[
    \varphi:\mathcal L\rightarrow \check S.
  \]
\end{theorem}
For maximal representations of $\Gamma_g$ into $\Sp(2n,\R)$ the theorem is proved in \cite{Anosov}.

Here we give a proof for Theorem \ref{ThmContLC} different from the one in \cite{S11}. In \cite{S11} we modified the proof for the continuity in \cite{Anosov} and parts of \cite{Surface} which where used there. Here we argue as follows:
\begin{proof}
Let $\varrho$ be representation of $\Sigma_{g,m}$ with $m\geq 1$ such that the $\varrho(C_i)$ are all $S$-hyperbolic. Since each real matrix is conjugate to its inverse, by Theorem \ref{PropGlueSp2nRIntro} $\varrho(C_i)$ is conjugate to its inverse. Therefore we can take a second copy of the same representation and glue the two copies to obtain a maximal representation $\bar \varrho$ of $\Sigma_{2g+m-1}$. We can consider $\Gamma_{g,m}$ as a subgroup of $\Gamma_{2g+m-1}$ with the property that $\bar \varrho|_{\Gamma_{g,m}}=\varrho$. The representation $\bar \varrho$ is a maximal representations of a closed surface group, hence it is Anosov and it has a continuous, equivariant limit curve $\bar \varphi:S^1\rightarrow \check S$, where the action on $S^1$ is given by a some hyperbolization $\bar h$ (\cite{Anosov}). The hyperbolization $\bar h$ restricts to a hyperbolization $h$ of $\Sigma_{g,m}$ with geodesic boundaries and limit set $\mathcal L$. The limit curve $\varphi$ of $\varrho$ is the restriction of $\bar \varrho$ to $\mathcal L$. Since $\bar \varphi$ is continuous so is its restriction $\varphi$. This also shows that $\varrho$ is Anosov.
\end{proof}

\subsection{Connected Components}\label{SecConComp}

Connected components of spaces of maximal representations of fundamental groups of closed surfaces have been counted using Higgs bundle techniques, see  \cite{RepSymplHiggs,GothenComp,DefMaxRep} and the references therein. Olivier Guichard and Anna Wienhard gave in \cite{GW09} one example for a representation in each connected component of $\Rep_{max}(\Gamma_g,\Sp(2n,\R))$. 

The parameters from Theorem \ref{MainThmIntro} allow us to count the connected components of $\Rep_{max}(\Gamma_{g,m},\Sp(2n,\R))$ for surfaces with boundary $\Sigma_{g,m}$ with $m\geq 1$. First we show that $\#\pi_0(\Rep_{max}(\Gamma_{g,m},\Sp(2n,\R)))\geq 2^{2g+m-1}$ (Proposition \ref{PropLowerBound}). Then we show, that one can deform every representation of $\Gamma_{g,m}$ with $m\geq 1$ into some standard representation (Proposition \ref{PropUpperBound}). This gives an upper bound and shows that $\# \pi_0(\Rep_{max}(\Gamma_{g,m},\Sp(2n,\R)))= 2^{2g+m-1}$.

 Before we start, we collect some fact needed later in this section.
\begin{lemma}\label{LemTwoComp}
  The set 
  \[ 
    B= \{X\in \GL(n,\R)| X \text{ contracting}\}
  \]
  has two connected components distinguished by the sign of the determinant.
\end{lemma}

Recall
\begin{align*}
  \tilde R_n=&\{(X_1,X_2,X_3)\in \GL(n,\R)^3|X_3(X_2^\top)^{-1}X_1 \text{ symmetric}\\
   &\text{and positive definite}\},\\
  R=&\{(X_1,X_2,X_3)\in \bar B^3|X_3(X_2^\top)^{-1}X_1 \text{ symmetric and positive definite} \}\\
   R^*=&\{(X_1,X_2,X_3)\in B^3|X_3(X_2^\top)^{-1}X_1 \text{ symmetric and positive definite} \}
\end{align*}
and
\begin{lemma}\label{LemCone2}
  Note that $(X_1,X_2,X_3)\in \tilde R_n$ if and only if 
  \[
   (\lambda_1 X_1,\lambda_2 X_2, \lambda_3 X_3)\in \tilde R_n
  \]
 for $(\lambda_1,\lambda_2,\lambda_3)\in \tilde R_1$. 
\end{lemma}

\begin{proposition}\label{PropCountR}
  The sets $\tilde R_n$, $R$ and $R^*$ have four connected components distinguished by $\sign \det X_i$, the signs of the determinants of the $X_i$.
\end{proposition}
\begin{proof}
  We begin with $\tilde R$. It can be identified with $\GL(n,\R)^2\times \Omega$, where $\Omega$ is the set of symmetric and positive definite matrices in $\GL(n,\R)$. Indeed, the map from  $\tilde R$ to $\GL(n,\R)^2\times \Omega$
  \[
    (X_1,X_2,X_3)\mapsto (X_1,X_2, X_3(X_2^\top)^{-1}X_1)
  \]
  is a homoeomorphism between these spaces. Since $\GL(n,\R)$ has two connected components distinguished by the signs of the determinants of $X_1$ and $X_2$, $\tilde R$ has four connected components ($\Omega$ is connected).   
  Now we proof the proposition for $R$. The main ingredient for the proof is Lemma \ref{LemCone2}.\\  
In any connected component of $\tilde R$ there is at least one connected component of $R$.  Indeed, let  $(X_1,X_2,X_3)\in \tilde R$ arbitrary, then there exists $\lambda_1$, $\lambda_2$ and $\lambda_3$ in $(0,1]$ such that 
$(\lambda_1 X_1,\lambda_2 X_2, \lambda_3 X_3)\in  R$. Hence $|\pi_0(R)|\geq |\pi_0(\tilde R)|$.
  \\
Now we show equality. Let $(X_1,X_2,X_3)$ and $(Y_1,Y_2,Y_3)$ be triples in $R$, such that there is a path $s=(s_1,s_2,s_3)$ joining them in $\tilde R$. By Lemma \ref{LemCone2} there exists for any $t$ a triple $(\lambda_1(t),\lambda_2(t),\lambda_3(t))\in (0,1]^3$ such that $(\lambda_1(t)s_1(t),\lambda_2(t)s_2(t),\lambda_3(t)s_3(t))\in R$ for all $t$. Since the image of $s$ is compact,  there exists $\lambda_1$, $\lambda_2$ and $\lambda_3$ such that the path $\tilde s(t):=(\lambda_1 s_1(t),\lambda_2 s_2(t), \lambda_3 s_3(t))$ is in $R$. It joins $(\lambda_1 X_1,\lambda_2 X_2, \lambda_3 X_3)$ and $(\lambda_1 Y_1,\lambda_2 Y_2,\lambda_3 Y_3)$ in $R$. Furthermore by construction there is a path joining $(X_1,X_2,X_3)$ and $(\lambda_1 X_1,\lambda_2 X_2, \lambda_3 X_3)$ as well as a path joining $(\lambda_1 Y_1,\lambda_2 Y_2,\lambda_3 Y_3)$ and $(X_1,X_2,X_3)$. Hence $(X_1,X_2,X_3)$ and $(Y_1,Y_2,Y_3)$ are in the same connected component of $R$ and $|\pi_0(R)|= |\pi_0(\tilde R)|$.

The proof for $R^*$ goes along the same lines.
\end{proof}

We use the notions from Theorem \ref{ThmGeneral} resp. Figure \ref{GenGluing}.
\begin{proposition}\label{RemComp}   The signs of the determinants of length parameters $X_i$ and $Y_j$ and twist parameters $H_k$ distinguish connected components.
\end{proposition}
\begin{proof}
  By Theorem \ref{ThmContLC} representations $\varrho$ with $\varrho(C_i)$ $S$-hyperbolic for all $i$ are Anosov. Since they are dense in the representation variety $\Rep(\Gamma_{g,m},\Sp(2n,\R)$, it is enough to prove the proposition for this case. For Anosov representations we can apply Lemma 4.11 in \cite{GW09}, which expresses the first Stiefel-Whitney class  of a certain bundle in terms of the representation $\varrho$. We use the notation from \cite{GW09}. For this lemma we need the interpretation of the Shilov boundary as the space of Lagrangian subspaces (c.f. Example \ref{ExShilovLagrangians}). We can assume $\xi(t_\gamma^s) = \langle e_{n+1},\ldots, e_{2n}\rangle$ and
 \[
   \varrho(\gamma)= \left(\begin{array}{cc}
	X&0 \\
	Y&(X^\top)^{-1}
\end{array}\right).
 \]
The matrix $\varrho(\gamma)$ acts on the last $n$ components of the vectors in $\xi(t_\gamma^s)$ by multiplication with $(X^\top)^{-1}$. 
Therefore we get by \cite[Lemma 4.11]{GW09} 
\[
  sw_1(\varrho)([\gamma])=\sign(\det \varrho(\gamma)|_{\xi(t^s_\gamma)})=\sign (\det (X^\top)^{-1})=\sign (\det X)
\]

For the twist parameter, we also have to consider elements of $\Sp(2n,\R)$ of the form
  \[
  \varrho(\bar\gamma)= \left(\begin{array}{cc}
	Y_3 GY_1^{-1}-(G^\top)^{-1}& Y_3G\\
	GY_1^{-1} & -G
\end{array}\right).
  \]
  We want to calculate $sw_1(\varrho)(\bar \gamma)$ as above. Hence we investigate the differential of $\varrho(\bar \gamma)$ in a fixed point. We know that $\varrho(\bar \gamma)$ is $S$-hyperbolic and that it has a pair of transversal fixed points $Y$ and $Y'$ with 
  \[
    \beta(Y_1,Y,0)=\beta(e,Y',\infty)=n.
  \]
  In particular $Y-Y_1$ is positive definite.  By construction $Y_1$ is negative definite, hence the sign of its determinant only depends on $n$. To obtain the derivative of $d$ in $Y$ we calculate
  \begin{align*}
   &  \left(\begin{array}{cc}
	1&-Y \\
	0&1
\end{array}\right) \left(\begin{array}{cc}
	Y_3 GY_1^{-1}-(G^\top)^{-1}&  Y_3G\\
	GY_1^{-1} & -G
\end{array}\right) \left(\begin{array}{cc}
	1&Y \\
	0&1
\end{array}\right)\\
=& \left(\begin{array}{cc}
	(Y_3 GY_1^{-1}-(G^\top)^{-1})-YGY_1^{-1} &0  \\
	GY_1^{-1} & GY_1^{-1} Y-G
\end{array}\right)
  \end{align*}
  Hence the same argument as above applies.
  
  It is enough to consider the length parameters $X_i$ and $Y_j$ and twist parameters $H_k$ because they already determine the Stiefel-Whitney classes. Indeed,  $sw_1(\varrho)(A_i)$, $ sw_1(\varrho)(B_i)$ and $ sw_1(\varrho)(C_j)$ are uniquely determined by them.
\end{proof}

Define: 
 \[
  X_\pm:=
  \left(\begin{array}{cccc}
	\pm \frac{1}{2}&&& \\
	&\frac{1}{2} &&\\
	&&\ddots &\\
	&&&\frac{1}{2}
\end{array}\right)\quad \text{and} \quad G_\pm:=
\left(\begin{array}{cccc}
	\pm 1&&& \\
	&1 &&\\
	&&\ddots &\\
	&&&1
\end{array}\right).
\]

\begin{proposition}\label{PropLowerBound}
  $\Rep_{max}(\Gamma_{g,m},\Sp(2n,\R))$ has at least $2^{2g+m-1}$ connected components if $m\geq 1$.
\end{proposition}
\begin{proof}
We can explicitly write down parameters for  $2^{2g+m-1}$ representations which are by Proposition \ref{RemComp} in different connected components. The length parameters of all of these representations are either $X_+$ or $X_-$. The twist parameter are either $G_+$ or $G_-$. 

We have a complete freedom of choice for the twist parameters $H_k$ between $G_+$ or $G_-$ as well as for all length parameter $X_i$ and $Y_j$ betweeen $X_+$ or $X_-$, except $X_1$. This leaves only one choice for all other length parameter if we want them to be either $X_+$ or $X_-$. By Proposition \ref{RemComp} they all lie in different connected compontens.
  \end{proof}

The representations associated with these parameters are twisted diagonal representations as defined in \cite{GW09}.
\begin{proposition}\label{PropUpperBound}
  Let $\varrho:\Gamma_{g,m}\rightarrow \Sp(2n,\R)$ be a maximal representation and we label its parameters as in Figure \ref{GenGluing}.  It can be deformed into a representation $\bar \varrho$ with length parameters 
  \[
     \bar L_i:= X_{\sign \det (L_i)} 
   \]
  where $L_i\in \{W_i,X_i,Y_i,Z_i\}$  is a length parameter from Figure \ref{GenGluing} and 
     \[ \bar T_j:=G_{\sign \det (T_j)},
  \]
   where $T_j\in \{G_j,H_j,J_j\}$ is a twist parameter and $\sign \det \in \{+,-\}$

  \end{proposition}
Together with Proposition \ref{PropLowerBound} we get:
\begin{theorem}\label{CorConComp}
    $\Rep_{max}(\Gamma_{g,m},\Sp(2n,\R))$ has $2^{2g+m-1}$ connected components if $m\geq 1$.
\end{theorem}

Before we can prove Proposition \ref{PropUpperBound} we need two lemmas, which are special cases. 

\begin{lemma}\label{LemDeformation03}
  Let $\varrho:\Gamma_{0,3}\rightarrow \Sp(2n,\R)$ be a maximal representation with three contracting  parameters $X_1$, $X_2$ and $X_3$.  Let $X_3(t):[0,1]\rightarrow B$  and $X_2(t):[0,1]\rightarrow B$ be two path starting in $X_1$ resp. $X_2$. Then there exists a path $X_1(t):[0,1]\rightarrow B$ such that $X_3(t)\big( X_2(t)^\top\big)^{-1}X_1(t)$ is symmetric and positive definite for any $t$, i.e. the path $(X_1(t),X_2(t),X_3(t))$ gives parameters for maximal representations of $\Gamma_{0,3}$ into $\Sp(2n,\R)$ for any $t$, such that all $X_i(t)$ are contracting. Hence it defines a path in $\Rep_{max}(\Gamma_{0,3},\Sp(2n,\R))$.
\end{lemma}
\begin{proof}
  Since $X_1$, $X_2$ and $X_3$ are parameters for a representation, we know that $S:=X_3(X_2^\top)^{-1}X_1$ is symmetric and positive definite. Defining $\tilde X_1(t):=(X_3(t)(X_2(t)^\top)^{-1})^{-1}S$, we get a path with $\tilde X_1(0)=X_1$ such that $\tilde X_1(t)$, $X_2(t)$ and $X_3(t)$ are parameters for a maximal representation for any $t\in [0,1]$. To fix the issue that $\tilde X_1(t)$ might be non-contracting for some $t$, we choose a curve $\lambda:[0,1]\rightarrow R_{>0}$ such that $\lambda(0)=1$ and $1/\lambda(t)$ is bigger than the absolute value of the biggest eigenvalue of $\tilde X_1(t)$ for all $t$. Now putting $X_1(t):=\lambda(t) \tilde X_1(t)$ finishes the proof.
\end{proof}

\begin{corollary}
  Every maximal representation of $\Gamma_{0,3}$ into $\Sp(2n,\R)$ with parameters $(X_1,X_2,X_3)$ can be deformed into the representation with parameters $(X_{\sign \det X_1},X_{\sign \det X_2},X_{\sign \det X_3})$.
\end{corollary}

\begin{lemma}\label{LemDeformation11}
  Let $\varrho:\Gamma_{1,1}\rightarrow \Sp(2n,\R)$ be a maximal representation with parameters $(X_1,X_2,G)$ as in Proposition \ref{PropG11}. Then $\varrho$ can be deformed into a maximal representation $\bar \varrho$  with parameters
  \[  
      \bar X_1=X_{\sign \det X_1},~ \bar X_2 =X_{\sign \det X_2} ,~ G=G_{\sign \det G}.
  \]
\end{lemma}
\begin{proof}
 Let $X_1(t)$ and $G(t)$ be paths in $\GL(n,\R)$ joining $X_1$ and $X_\pm$ resp. $G$ and $G_\pm$. Choose a path $S(t)$ in the symmetric positive definite matrices joining $GX_2^\top G^{-1}(X_2^\top)^{-1} X_1 $ and  $\frac{1}{2}I$ such that $X_1(t):= (G(t)X(t)_2^\top G(t)^{-1}(X(t)_2^\top){-1} )^{-1}S(t)$ is contracting for all $t$. This is possible since one can scale $S(t)$ as in the proof of Proposition \ref{LemDeformation03} such that the eigenvalues of $X_1(t)$ are small enough.  These paths defines a path in $\Rep_{max}(\Gamma_{1,1},\Sp(2n,\R))$ joining $\varrho$ with the desired representation.
\end{proof}

\begin{proof}[Proof of Proposition \ref{PropUpperBound}]
  We prove Proposition \ref{PropUpperBound} by recurrence. By Lemma \ref{LemDeformation03} and Lemma \ref{LemDeformation11} it is true for representations of $\Gamma_{0,3}$ and $\Gamma_{1,1}$. So we assume now that it is true for $\Gamma_{g,m}$ with $m\geq 1$. 
  \begin{enumerate}
     \item The case $\Gamma_{g,m+1}$\\
     Let $\varrho:\Gamma_{g,m+1}\rightarrow \Sp(2n,\R)$ be a maximal representation. 
     First note that we obtain $\Sigma_{g,m+1}$ by gluing $\Sigma_{g,m}$ and $\Sigma_{0,3}$. Let $(X_1,X_2,X_3)$ be  parameters of the restriction of $\varrho$ to $\pi_1(\Sigma_{0,3})$.
     
     By recurrence assumption we can deform the parameters of $\varrho|_{\Gamma_{g,m}}$ as requested. This produces a path $X(t)$ for the boundary component of $\Sigma_{g,m}$ along which we can glue $\Sigma_{g,m}$ and $\Sigma_{0,3}$. Joining the twist parameter for this boundary component with $G_+$ resp. $G_-$ defines a path from $X_3$ to $X_+$ resp. $X_-$. By Lemma \ref{LemDeformation03} this  path, together with a path $X_2(t)$ joining $X_2$ and $X_+$ resp. $X_-$ produces a path which joins $\varrho$ with the desired representation.    
    \item The case $\Gamma_{g+1,m}$\\
    This case works analogously by gluing a surface $\Sigma_{1,2}$ to $\Sigma_{g,m}$.
  \end{enumerate}
\end{proof}

\bibliography{PaperFenchelNielsen}

\end{document}